\documentclass[12pt, reqno]{amsart}

\pdfoutput=1

\usepackage[utf8]{inputenc}
\usepackage[T1]{fontenc}
\usepackage{lmodern}
\usepackage{amsthm, amssymb, amsmath, amsfonts, mathrsfs}

\usepackage{microtype}

\usepackage[final]{changes}

\usepackage[pagebackref,colorlinks=true,pdfpagemode=none,urlcolor=blue,
linkcolor=blue,citecolor=blue]{hyperref}

\usepackage{amsmath,amsfonts,amssymb,amsthm}
\usepackage{amsthm, amssymb, amsmath, amsfonts, mathrsfs}

\usepackage{mathrsfs}
\usepackage{MnSymbol}
\usepackage{scalerel} % for scaling the cube

\usepackage{color}
\usepackage{accents}

\usepackage{bbm}

\definecolor{labelkey}{gray}{.8}
\definecolor{refkey}{gray}{.8}

\definecolor{darkred}{rgb}{0.9,0.1,0.1}
\definecolor{darkgreen}{rgb}{0,0.5,0}

\newtheorem{theorem}{Theorem}[section]
\newtheorem{lemma}[theorem]{Lemma}

\newtheorem{corollary}[theorem]{Corollary}
\newtheorem{proposition}[theorem]{Proposition}

\theoremstyle{remark}
\newtheorem{remark}[theorem]{Remark}

\renewenvironment{proof}[1][Proof]{ {\itshape \noindent {#1.}} }{$\Box$
\medskip}

\numberwithin{equation}{section}
\numberwithin{figure}{section}
\theoremstyle{plain}

\newcommand{\R}{\mathbb{R}}

\newcommand{\1}{\mathbbm{1}}
\newcommand{\eps}{\varepsilon}

\newcommand{\dd}{\mathrm{d}}

\newcommand{\Var}{\mathrm{Var}}
\newcommand{\Cov}{\mathrm{Cov}}

\newcommand{\Prob}{\mathbb{P}}

\newcommand{\E}{\mathbb{E}}
\newcommand{\EE}{\mathrm{E}}
\newcommand{\EEE}{\mathbf{E}}

\begin{document}
\title[Mesoscopic averaging of $2d$ KPZ]{Mesoscopic averaging of the two-dimensional KPZ equation}

\author{Ran Tao}
\address[Ran Tao]{Department of Mathematics, University of Maryland, College Park, MD, 20740.  rantao16@umd.edu}

\begin{abstract}
We study the limit of a {local average} of the KPZ equation in dimension $d=2$ with general initial data in the subcritical regime. Our result shows that a proper spatial averaging of the KPZ equation converges in distribution to the sum of the solution to a deterministic KPZ equation and a Gaussian random variable that depends solely on the scale of averaging. This shows a unique mesoscopic averaging phenomenon that is only present in dimension two. Our work is inspired by the recent findings by Chatterjee \cite{chatterjee2023weak}.
\bigskip


\end{abstract}

\maketitle

\section{Introduction and Main result}\label{s.introduction}
\subsection{Main Result}
We are interested in the two dimensional KPZ equation driven by a mollified space-time white noise
\begin{equation}
\partial_t h^\eps (t,x) = \frac{1}{2} \Delta h^\eps (t,x) + \frac{1}{2} |\nabla h^{\eps}(t,x)|^2+ \frac{\beta}{\sqrt{\log \eps^{-1}}}\xi_\eps(t,x) - C_\eps,
\label{eq:kpz2m}
\end{equation}
with initial condition $h^\eps (0,x) = h_0(x)$ and $C_\eps =  \frac{\beta^2}{{2}\eps^2\log \eps^{-1}}\|\phi \|_{L^2(\R^2)}^2$.
Here $\beta>0$ is a constant. We define $\xi_{\eps}(t,x) = \phi^\eps * \xi (t,x)$, where $\xi$ is the space-time white noise, and $\phi \in C_c^\infty(\R^2)$ is a non-negative, smooth, symmetric mollifier with $\int{\phi} \dd x =1$ and $\phi^\eps (x) = \frac{1}{\eps^2}\phi(\eps^{-1}x)$. We use $*$ to denote convolution in space. The noise $\xi_{\eps}$ is white in time and colored in space, with spatial correlation length in the scale of $\eps$. 

Inspired by Chatterjee's recent work \cite{chatterjee2023weak}, we will investigate the limit of a \textit{local average} of the mollified KPZ equation \eqref{eq:kpz2m}  when $\beta \in (0, \sqrt{2\pi})$. We will show that when $h^{\eps}(t,x)$ is averaged in space in a proper scale, its limit would equal in distribution to the sum of the solution to a deterministic KPZ equation and a Gaussian random variable. The deterministic part only depends on the initial condition $h_0$. The Gaussian random variable only depends on the scale of averaging. Our result shows a ``mesoscopic'' averaging phenomenon that appears exclusively in dimension two.

The following theorem is the main result of the paper.

\begin{theorem}\label{thm:main}
Let $h^{\eps} (t,x)$ be the solution to the mollified 2-dimensional KPZ equation \eqref{eq:kpz2m}
with the initial condition $h_0 : \R^2 \to \R$ being bounded and Lipschitz continuous. Let $0< \beta < \sqrt{2\pi}$ and {$r_{\eps} = \eps^{1-\gamma}$} for some $0\leq \gamma \leq 1$. We use  $B(x, r_{\eps})$ to denote a ball with center $x$ and radius $r_{\eps}$.

For any fixed $t >0$, $x \in \R^2$, the local average of $h^{\eps}$ over $B(x, r_{\eps})$ converges as follows 
\begin{equation}
\begin{aligned}
    &\frac{1}{|B(x, r_\eps)|}\int_{|y-x|\leq{r_\eps}}h^{\eps} (t,y) \dd y\\
    & \xrightarrow{d}  \mathcal{N}(0, \sigma_{\gamma}^2) - \frac{1}{2}\log\left(\frac{2\pi}{2\pi - \beta^2}\right)+ \begin{cases}
\bar{h}(t,x), & \quad \text{if } 0 \leq  \gamma <1, \\
\frac{1}{|B(x, 1)|}\int_{|y-x|\leq{1}} \bar{h}(t,y) \dd y, & \quad \text{if } \gamma = 1,
\end{cases}\label{eq:limit}
\end{aligned}
\end{equation}
as $\eps \to 0$, where $\bar{h}(t,x)$ is the solution to the 2-dimensional {deterministic} KPZ equation
\begin{equation}
    \partial_t \bar{h}(t,x) = \frac{1}{2}\Delta \bar{h}(t,x) + \frac{1}{2} |\nabla \bar{h}(t,x)|^2, \quad \bar{h}(0,x) = h_0(x), \label{eq:deterkpz}
\end{equation}
and $\mathcal{N}(0, \sigma_{\gamma}^2) $ is a normal distribution with mean 0 and variance $\sigma_{\gamma}^2 = \log\left( \frac{2\pi - \beta^2\gamma}{2\pi - \beta^2}\right)$. When $0 \leq \gamma <1$, we have $\sigma_{\gamma}^2 >0$.  We treat $\mathcal{N}(0, 0)$ as constant zero, when $\gamma =1$ and $\sigma_{\gamma}^2 =0$.
\end{theorem}
We first make the following remarks.

If $\gamma = 0$, the averaging was performed over a ball with radius $r_\eps  = \eps$. The limit in distribution is \[\bar{h}(t,x) + \mathcal{N}(0,  \log \frac{2\pi}{2\pi - \beta^2}) - \frac{1}{2}\log\left(\frac{2\pi}{2\pi - \beta^2}\right).\] If we take $h_0 =0$, then $\bar{h}(t,x) = 0$ and the limit coincides with the point-wise limit of $h^{\eps}(t,x)$ as $\eps \to 0$ obtained in \cite{caravenna2017universality} (see Theorem~\ref{thm:kpzpw} below). In fact, our result generalizes the point-wise limit of KPZ equation with flat initial data to KPZ equation with more general (bounded and Lipschitz continuous) initial conditions. The case $\gamma = 0$ can be viewed as a ``microscopic'' averaging result.

If $\gamma = 1$, the averaging was performed over a ball with radius $r_\eps =1$. As shown in \eqref{eq:limit}, the Gaussian term equals zero as $\sigma_\gamma^2 = 0$. The limit would be deterministic. It equals to a spatial average of the deterministic KPZ equation over the ball of radius 1, plus a height shift $- \frac{1}{2}\log\left(\frac{2\pi}{2\pi - \beta^2}\right)$. 

The randomness has disappeared due to the independence in the limit of $h^\eps$ for distinct points. As we will see in Theorem~\ref{thm:kpzpw} below, a result of Caravenna-Sun-Zygouras, for any finite set of distinct points $(x_i)_{1\leq i \leq n}$, the random variables $(h^{\eps}(t, x_i))_{1\leq i \leq n}$ converge to independent Gaussians as $\eps \to 0$. Thus taking a local average over a fixed-radius ball eliminates the randomness as a result of  the law of large numbers. The height shift $- \frac{1}{2}\log\left(\frac{2\pi}{2\pi - \beta^2}\right)$ is the mean of the Gaussians. The case $\gamma = 1$ can be viewed as a ``macroscopic'' averaging result.

{When $\gamma = 1$, to study the next order random fluctuations, one should look at the error $h^{\eps}(t,x) - \E h^{\eps}(t,x)$ with amplification. \cite{chatterjee2020constructing} started such studies. \cite{caravenna2020two} (as well as \cite{gu2020gaussian} for a smaller regime of $\beta$) proved that, after proper rescaling, the error converges to an Edwards-Wilkinson limit. \cite{nakajima2023fluctuations} later enhanced this result for more general initial conditions and multi-dimensional parameters. Although we both derive Gaussian limits, our work here differs from previous studies. Our study focuses on the first order term $\int_{B(0,r_\eps)} h^{\eps}(t,x) \dd x$, whereas the previous works all studied the second order fluctuations $\sqrt{\log \eps^{-1}}\int_{B(0,1)}  h^{\eps}(t,x) - \E h^{\eps}(t,x)  \dd x$. }

If $0 < \gamma <1$, the averaging was performed over a ball with radius $r_\eps = \eps^{1-\gamma}$, where $ \eps \ll r_\eps \ll 1$. The limit in distribution is the deterministic KPZ equation with the same height shift, plus a Gaussian random variable. The variance of the Gaussian, $\sigma^2_\gamma$, is dependent solely on $\gamma$. This is  an interpolation between the limiting case $\gamma = 0$ and $\gamma =1$ and is a ``mesoscopic'' averaging result.

The choice of radius $r_\eps = \eps^{1-\gamma}$ is motivated by the work of Chatterjee \cite{chatterjee2023weak} in dimensions three and higher and also by the multi-dimensional limit of $h^\eps(t,x)$ in Theorem~\ref{thm:kpzpw} below. For 
a finite collection of space-time points $( t, x_\eps^{(i)})_{1 \leq i \leq n}$, 
 $(h^\eps(t, x_\eps^{(i)}))_{1 \leq i \leq n}$ converges in joint distribution to a multi-dimensional normal distribution with
a covariance matrix depending on the power scales of  $\{ |x_\eps^{(i)} - x_\eps^{(j)}|: 1 \leq i < j \leq n\}$. When $|x_\eps^{(i)} - x_\eps^{(j)}| = \eps^{1-\gamma+o(1)}$, $\Cov(h^\eps(t, x_\eps^{(i)}), h^\eps(t, x_\eps^{(j)})) \to \sigma_\gamma^2$ as $\eps \to 0$. Since an average of Gaussians would remain a Gaussian, the limit of $h^\eps$ averaged over a ball of radius $ \eps^{1-\gamma}$ would be a Gaussian with variance $\sigma^2_\gamma$.

The ``height shift'' term $- \frac{1}{2}\log\left(\frac{2\pi}{2\pi - \beta^2}\right)$ equals to $\lim_{\eps \to 0} \E h^\eps(t, x)$ with initial condition $h_0=0$ for any $t>0$ and $x \in \R^2$. This term appears because we take logarithm in solving the KPZ equation and it is also derived in the point-wise limit in Theorem~\ref{thm:kpzpw}.

\begin{remark}\label{rm:fieldlimit}
Another interpretation of Theorem~\ref{thm:main} is to view the convergence from the perspective of generalized random fields (also known as \textit{random distribution} in the literature) in microscopic variables.
For any $(t,x) \in [0, +\infty) \times \R^2$, we consider the following microscopic equation
\begin{equation*}
\partial_t \tilde{h}^\eps(t,x) = \frac{1}{2}\Delta \tilde{h}^\eps(t,x) + \frac{1}{2} |\nabla \tilde{h}^\eps(t,x)|^2 + \frac{\beta}{\sqrt{\log \eps^{-1}}} \xi_1(t,x) - \frac{\beta^2}{2\log \eps^{-1}}\|\phi\|_{L^2(\R^2)}^2,
\label{eq:microvar}
\end{equation*}
with initial condition $\tilde{h}^\eps(0,x) = h_0(\eps x)$. The above $\xi_1$ is defined as in \eqref{eq:kpz2m} with $\eps =1$. By the scaling property of the space-time white noise,
\[
h^\eps(\cdot,\cdot) = \tilde{h}^\eps(\frac{\cdot}{\eps^2},\frac{\cdot}{\eps}) \qquad \text{jointly in law}.
\]

The proof of Theorem~\ref{thm:main} can be modified to generalize the following result:

Fix $t>0$ and $x \in \R^2$. Let $0 \leq  \gamma <1$ and $0< \beta <\sqrt{2\pi}$.  For any smooth and compactly supported test function $g \in C_c^{\infty}(\R^2)$,  
\[
  \int_{\R^2}\tilde{h}^{\eps} (\frac{t}{\eps^2},\frac{x}{\eps}+\frac{y}{\eps^\gamma}) g(y) \dd y \xrightarrow{d}  [\mathcal{N}(0, \sigma_{\gamma}^2) - \frac{1}{2}\log\left(\frac{2\pi}{2\pi - \beta^2}\right)+ \bar{h}(t,x)]\int_{\R^2}g(y)\dd y,
\]
as $\eps \to 0$.  

By \cite[Corollary 2.4]{bierme2018generalized}, as a generalized random field, $\tilde{h}^{\eps} (\frac{t}{\eps^2},\frac{x}{\eps}+\frac{\cdot}{\eps^\gamma})$ converges in law to the constant Gaussian random field $[\mathcal{N}(0, \sigma_{\gamma}^2) - \frac{1}{2}\log\left(\frac{2\pi}{2\pi - \beta^2}\right)+ \bar{h}(t,x)] \1_{\R^2}(\cdot)$. This constant Gaussian field on $\R^2$ has a fixed spatial covariance $\sigma_\gamma^2$ independent of the spatial variable. This limit is distinct from other Gaussian random field limits found in the literature on KPZ equation.
\end{remark}

One can also consider the locally averaged KPZ equation as a generalized random field on $\R^2$. Let $0 \leq \gamma < 1$. Define 
\[
\mathfrak{h}^{\eps, \gamma}(t,x) :=  \frac{1}{|B(x, r_\eps)|}\int_{|y-x|\leq{r_\eps}}h^{\eps} (t,y) \dd y, \quad \text{with }r_\eps = \eps^{1-\gamma}.
\]
The following corollary shows that, as a generalized random field, $\mathfrak{h}^{\eps, \gamma}(t,\cdot)$ converges to \added{the 2-dimensional {deterministic} KPZ limit} $\bar{h}(t, \cdot)- \frac{1}{2}\log\left(\frac{2\pi}{2\pi - \beta^2}\right)$ as $\eps \to 0$.

\begin{corollary}\label{co:localfield}
Take any smooth and compactly supported test function $g \in C_c^{\infty}(\R^2)$. For any $t >0$ and $0 \leq \gamma < 1$, as $\eps \to 0$, the random variable $\int_{\R^2}\mathfrak{h}^{\eps, \gamma}(t,x) g(x)\dd x$ converges in probability to $\int_{\R^2} \left[\bar{h}(t, x)- \frac{1}{2}\log\left(\frac{2\pi}{2\pi - \beta^2}\right)\right] g(x)\dd x$.
\end{corollary}

\begin{remark}\label{rm:nextstep}
A natural question to ask next is to investigate the second order fluctuations of $\mathfrak{h}^{\eps, \gamma}(t,\cdot)$ after proper rescaling. Inspired by the results in \cite{caravenna2020two, gu2020gaussian} and \cite{nakajima2023fluctuations}, one may expect that $\sqrt{\log \eps^{-1}}\int_{\R^2} [\mathfrak{h}(t, x) - \E \mathfrak{h}(t, x)]g(x) \dd x$ has a non-degenerate Gaussian limiting distribution as $\eps \to 0$. We do not pursue further in this direction, as proving such results may require techniques that are beyond the scope of this paper. 
\end{remark}

We now conclude the section with a note on the significance of this study.
In Theorem~\ref{thm:main}, we proved that the local average of the KPZ equation in dimension two shows a specific mesoscopic averaging phenomenon. 
In Remark~\ref{rm:fieldlimit}, when the local average is understood in a microscopic scale, we obtained a generalized random field converge in law to a constant Gaussian random field, which is novel to the literature. 
Additionally, we showed in Corollary~\ref{co:localfield} that the locally averaged KPZ equation, as a generalized random field, converges to a deterministic KPZ limit in dimension two.
Our work studies the two-dimensional KPZ equation from a new perspective and it improves the understanding on this subject.

\subsection{Context}
On $(t,x) \in [0, +\infty) \times \R^d$, the KPZ equation is an SPDE formally given by 
\begin{equation}
    \partial_t h (t,x) = \frac{1}{2}\Delta h (t,x) + \frac{1}{2} |\nabla {h}(t,x)|^2 +  \xi (t,x). \label{eq:kpz}
\end{equation}
Here $\xi$ denotes the space-time white noise which is the distribution valued Gaussian field in spatial dimension $d$
with the covariance function  
\begin{equation*}
	\mathbb{E}\left[\xi(t,x)\xi(t',x')\right]=\delta_0(t-t')\delta_0(x-x').\nonumber\label{eq:covarianceoperator}
\end{equation*}
Here $\delta_0$ is the Dirac mass.

The KPZ equation was first introduced in 1986 by Kardar, Parisi and Zhang \cite{kardar1986dynamic} and has since become the standard random interface growth model in physics. However, mathematically, the KPZ equation is ill-posed due to the non-linear term $\nabla h$ being a generalized function, which makes the interpretation of $|\nabla h|^2$ unclear.

In 1990's, Bertini and Giacomin \cite{bertini1997stochastic} formulated the Hopf-Cole solution of \eqref{eq:kpz} in $d=1$ by a transformation $h (t,x) = \log u (t,x)$. In fact, let $u(t,x)$ be the solution to the $d$-dimensional stochastic heat equation (SHE) on $(t,x) \in [0, +\infty) \times \R^d$
\begin{equation}
\partial_t u(t,x) = \frac{1}{2}\Delta u(t,x) + u(t,x) \xi(t,x). \label{eq:she}
\end{equation}
Then at least formally, if ignoring the Itô correction, $\log u(t,x)$ satisfies \eqref{eq:kpz}.

However, the stochastic heat equation \eqref{eq:she} is only well-posed for $d=1$. For dimension $d \geq 2$,  \eqref{eq:she} does not have function or distribution valued solutions. This makes the problem of stochastic heat equation and KPZ equation in higher dimensions more challenging to solve. 

In \cite{caravenna2017universality}, Caravenna, Sun and Zygouras introduced a space-regularized equation with a scaling on the disorder strength to address the well-posedness issue in dimension $d=2$. The equation is given by
 \begin{equation}
\partial_t u^\eps (t,x) = \frac{1}{2} \Delta u^\eps (t,x) + \frac{\beta}{\sqrt{\log \eps^{-1}}}u^\eps(t,x)\xi_\eps(t,x),  \qquad u^{\eps}(0,x) = u_0(x).
\label{eq:she2m}
\end{equation}
Through the Hopf-Cole transformation $h^\eps(t,x) = \log u^\eps (t,x)$,  $h^{\eps}$ solves the mollified KPZ equation \eqref{eq:kpz2m} with $h_0 = \log u_0$.

In \cite{caravenna2017universality}, the following multi-scale point-wise asymptotic limit of $h^{\eps}(t,x)$ with initial condition $h_0(x)=0$ was derived.

\begin{theorem}[{\cite{caravenna2017universality}, Theorem 2.15, Remark 2.16}]\label{thm:kpzpw}
Let $h_0(x) = 0$. Fix $t>0$. Consider a finite collection of space-time points $( t_{\eps}, x_\eps^{(i)})_{1 \leq i \leq n}$, where $t_{\eps}>0$, $x_\eps^{(i)} \in \R^2$, such that as $\eps \to 0$, $t_\eps = \eps^{o(1)}t$ and
\[
\forall i, j \in \{1,\dots, n\}:  \quad |x_\eps^{(i)} - x_\eps^{(j)}| \leq \eps^{(1-\zeta_{i,j})+o(1)} \qquad \text{for some } \zeta_{i,j} \in [0,1].
\]

Then if $\beta \in (0, \sqrt{2\pi})$, $(h^\eps(t_\eps, x_\eps^{(i)}))_{1 \leq i \leq n}$ converges in joint distribution to the multi-dimensional normal distribution $(Y_i - \frac{1}{2}\Var[Y_i])_{1\leq i \leq n}$ as $\eps \to 0$, where $(Y_i)_{1\leq i \leq n}$ are jointly Gaussian random variables with
\[
\E[Y_i] = 0, \qquad \Cov[Y_i,Y_j] = \log\left( \frac{2\pi - \beta^2 \zeta_{i,j}}{2\pi-\beta^2}\right).
\]
If $\beta \geq \sqrt{2\pi}$, $h^\eps(t_\eps,x_\eps^{(i)})$ converges to $-\infty$ in probability for all $1\leq i \leq n$, as $\eps \to 0$.
\end{theorem}
If $\beta \in (0, \sqrt{2\pi})$, $\Var [Y_i] = \log \left(\frac{2\pi}{2\pi-\beta^2}\right)$ as $\zeta_{i,i} = 0$ for any $1 \leq i \leq n$. For any finite set of distinct points $(x_i)_{1\leq i \leq n}$, the random variables $(h^{\eps}(t_\eps, x_i))_{1\leq i \leq n}$ converge to independent Gaussians as $\eps \to 0$. This is because
 $\zeta_{i,j} = 1$ for all $1 \leq i, j \leq n$ with $i \neq j$, implying independence in the limit. 

\begin{remark}
{The value $\beta_c = \sqrt{2\pi}$ is critical here as there is a phase transition in Theorem~\ref{thm:kpzpw}. The interval $(0, \sqrt{2\pi})$ is known as the \textit{subcritical regime} for 2-dimensional KPZ equation.} In our paper, we will only focus on this regime. The research on the 2-dimensional SHE \eqref{eq:she} in a critical window around $\beta_c$ initiated with \cite{bertini1998two}, and notable recent advancements have been made in \cite{caravenna2019dickman, caravenna2023critical, caravenna2023critical2, caravenna2023quasi, gu2021moments, caravenna2019moments}.
\end{remark}

The directed polymer model in the random environment in $2+1$ dimension is related to the two-dimensional KPZ equation. The solution to the KPZ equation equals in law to the logarithm of the partition function of a \replaced{continuum }{continuous }directed polymer. (See \eqref{eq:kpzfk} below.)

In a prior work, Chatterjee \cite{chatterjee2023weak} proved that a growing random surface generated by discrete directed polymers in $d \geq 3$ converges to a deterministic KPZ equation. It is an analogue of our Theorem~\ref{thm:main} for directed polymers in dimensions $d+1$ with $d \geq 3$. The same result as in  \cite{chatterjee2023weak} is expected to hold in the \replaced{continuum }{continuous }setting, i.e. for the analogue KPZ equation in dimension $d\geq 3$ defined as
\[
\partial_t \hat{h}_\eps(t,x) = \frac{1}{2}\Delta \hat{h}_\eps(t,x)+\frac{1}{2}|\nabla \hat{h}_\eps(t,x)|^2 + \beta{\eps^{\frac{d-2}{2}}}\hat{\xi}_\eps(t,x) -  \frac{\beta^2}{{2}\eps^2}\|\hat{\phi} \|_{L^2(\R^d)}^2, \quad \hat{h}_\eps(0,x) = \hat{h}_0(x),
\]
where $\hat{\xi}_\eps = \xi(t,x) * \hat{\phi}^\eps$ is the space-time white noise in dimension $d+1$ spatially smoothened at scale $\eps$, $\hat{\phi} \in C_c^\infty(\R^d)$ is a non-negative, smooth, symmetric mollifier with $\int{\hat{\phi}} \dd x =1$ and $\hat{\phi}^\eps (x) = \frac{1}{\eps^d}\hat{\phi}(\eps^{-1}x)$, $\hat{h}_0$ is bounded and Lipschitz continuous, and $\beta$ is sufficiently small.

We note that there is a significant difference between the $d=2$ KPZ equation and the $d \geq 3$ KPZ equation. In dimension $d \geq 3$, there is no ``mesoscopic'' averaging phenomenon. In fact, in $d\geq 3$, by \cite[Theorem 2.2]{chatterjee2023weak}, we expect that, by taking the average of $\hat{h}_\eps(t,x)$ over a ball $B(x,r_\eps)$ with radius $\eps \ll r_\eps \ll 1$, the limit would be the corresponding deterministic KPZ equation with a height shift. The ``mesoscopic'' averaging result (when $ \eps \ll r_\eps \ll 1$ in our Theorem~\ref{thm:main}) is exclusive to dimension $d=2$.

The difference arises as follows. 
For $d \geq 3$, if $x_\eps^{(1)}, x_\eps^{(2)}\in \R^d$ and $\eps \ll |x_\eps^{(1)}-x_\eps^{(2)}| \ll 1$, then as $\eps \to 0$, $\hat{h}_\eps(t, x_\eps^{(1)})$ and $\hat{h}_\eps(t, x_\eps^{(2)})$ always become asymptotically independent, regardless of the scale of $|x_\eps^{(1)}-x_\eps^{(2)}|$. Taking a spatial average over a ball with radius $\eps \ll r_\eps \ll 1$ eliminates the ``randomness'' due to the law of large numbers. However, in dimension $d=2$, as noted in Theorem~\ref{thm:kpzpw}, when $\zeta_{i,j} \neq 0$, there is a nontrivial multi-scale correlation in the limit of $(h^{\eps}(t,x_i))_{1\leq i \leq n}$. 

We shall finish this section by mentioning some other relevant recent works. 
The second order fluctuations of KPZ equation in dimension $d=2$, i.e. the limit of $\sqrt{\log \eps^{-1}}[h^{\eps}(t,\cdot) - \E h^{\eps}(t,\cdot)]$ as $\eps \to 0$, was studied in the aforementioned works \cite{chatterjee2020constructing, caravenna2020two, gu2020gaussian, nakajima2023fluctuations}.
The second order fluctuations of SHE and nonlinear SHE in $d=2$ were studied in \cite{caravenna2017universality, nakajima2023fluctuations} and \cite{tao2022gaussian} respectively. \cite{gabriel2021central} studied the macroscopic-level limit of the polymer paths in dimension $2+1$. In $d \geq 3$, the point-wise limit of SHE and KPZ equation were studied in \cite{mukherjee2016weak, comets2020renormalizing}. The second order fluctuations of SHE and KPZ equation were studied in \cite{comets2019space, cosco2022law, cosco2021gaussian, dunlap2020fluctuations,  gu2018edwards, lygkonis2022edwards, magnen2018scaling}. The second order fluctuations of nonlinear SHE was studied in \cite{gu2020nonlinear}.

Studies of the deterministic KPZ equation and deterministically growing surfaces were recently conducted in \cite{chatterjee2022convergence, chatterjee2021universality}. By dropping the random noise from the environment, the deterministic KPZ equation is a much simpler object and is well-defined in all dimensions. The hope of such studies was that it could lead to an understanding of the universal nature of KPZ growth with noise in the dimensions greater than one. 

\subsection{Sketch of proof}
Before beginning the proof of Theorem~\ref{thm:main}, let us first outline its central idea. 

To study $h^{\eps}$, we analyze $u^{\eps} = e^{h^{\eps}}$. To do so, we employ the Feynman-Kac formula from \cite{bertini1995stochastic}, which can be easily adapted to any dimension.

Let $0 < \inf_{x\in \R^2} u_0(x) \leq \sup_{x\in \R^2} u_0(x) < \infty$, which is equivalent to the requirement that $h_0(x)$ is bounded. The Feynman-Kac formula states that the solution to the 2-dimensional stochastic heat equation \eqref{eq:she2m} is given by:
\begin{equation}
 u^\eps(t,x)=
\EE_x\left[u_0(B_{t}) \exp \left\{\beta_\eps \int_0^t \xi_{\eps}(t-r, B_r) \dd r -\frac{\beta_{\eps}^2}{2}  \E\left[\left(\int_0^t \xi_{\eps}(t-r, B_r) \dd r \right)^2\right] \right\}\right],
\label{eq:fk1}
\end{equation}
where $\beta_{\eps} := \frac{\beta}{\sqrt{\log \eps^{-1}}}$, $\EE_x$ is the expectation w.r.t. $(B_r)_{r\geq 0}$, $(B_r)_{r\geq 0}$ is a standard Brownian motion in $\R^2$ starting from $B_0=x$ and $\E$ denotes the expectation with respect to the space-time white noise in the environment. Later, we will also make use of the Brownian bridge from $(0, x)$ to $(s, y)$ with $s >0$ and $x, y \in \R^2$. We denote the expectation w.r.t. such Brownian bridge by $\EE_{0, x}^{s, y}$.

By a time reversal in $\xi_{\eps}$, a scaling property of the space-time white noise in dimension {${d=2}$}, and the scaling invariance of {Brownian} motion ($\eps B_{{\eps^{-2}}r} \stackrel{law}{=} B_r $), we find that $\{u^{\eps}(t,x)\}_{x \in \R^2}$ in \eqref{eq:fk1} has the same joint distribution as $\{\tilde{u}^{\eps}(t,x)\}_{x \in \R^2}$, where
\begin{equation*}
\begin{aligned}
 \tilde{u}^\eps(t,x)&=
\EE_x\Big[u_0(B_t)\exp \Big\{\beta_{\eps}\int_0^t \xi_{\eps}(r, B_r) \dd r -\frac{\beta_{\eps}^2}{2} \E\big(\int_0^t \xi_{\eps}(r, B_r) \dd r \big)^2 \Big\}\Big]\\
&=
\EE_x\Big[u_0(B_t)\exp \Big\{\beta_{\eps}\int_0^t  \int_{\R^2} \phi^{\eps}(B_r-y)\xi(r, y) \dd y\dd r -\frac{\beta_{\eps}^2}{2} t \|\phi^{\eps}\|_{L^2(\R^2)}^2\Big\}\Big]\\
&= \EE_{\eps^{-1}x}\Big[u_0(\eps B_{\eps^{-2}t})\exp \Big\{\beta_{\eps}\int_0^{\eps^{-2}t} \int_{\R^2} \phi(B_{\tilde{r}}-\tilde{y})\tilde{\xi}(\tilde{r}, \tilde{y}) \dd \tilde{y}\dd \tilde{r}-\frac{\beta_{\eps}^2}{2} \eps^{-2}t \|\phi\|_{L^2(\R^2)}^2\Big\}\Big].
\end{aligned}\label{eq:fk2}
\end{equation*}
In the last step, we made the change of variable $(\eps \tilde{y}, \eps^2 \tilde{r}) := (y, r)$. The scaling property of the space-time white noise implies that
\[
\tilde{\xi}(\tilde{r}, \tilde{y}) \dd \tilde{y}\dd \tilde{r} := \eps^{-2} {\xi}(\eps^2\tilde{r}, \eps\tilde{y}) \dd (\eps\tilde{y}) \dd (\eps^2 \tilde{r})
\]
is another space-time white noise in $\R^2$.

Now since $h^{\eps} = \log{u^{\eps}}$ and $h_0 = \log u_0$, we have 
\begin{equation}
\begin{aligned}
 h^\eps(t,x)&\stackrel{law}{=} 
\log \EE_{\eps^{-1}x}\Big[\exp \Big\{h_0(\eps B_{\eps^{-2}t})+\beta_\eps\int_0^{\eps^{-2}t} \int_{\R^2} \phi(B_r-y)\xi(r, y) \dd y \dd r
-\frac{1}{2}\beta_\eps^2 \eps^{-2}t \|\phi\|_{L^2(\R^2)}^2\Big\}\Big].
\end{aligned}\label{eq:fkkpz}
\end{equation}
For any $s \geq 0$ and $f \in {C}(\R^2)$ being bounded and Lipschitz continuous, define
\begin{equation}
\begin{aligned}
\Psi_{s}^{\beta_{\eps}, f} &=\Psi_{s}^{\beta_{\eps}, f} (B, \xi) = \exp \left[f (\eps B_{s}) + \beta_{\eps}\int_0^s \int_{\R^2} \phi(y-B_r)\xi(r, y) \dd y \dd r-\frac{1}{2}\beta_{\eps}^2 \|\phi\|_{L^2}^2s\right].
\end{aligned}\label{eq:psi}
\end{equation} 
Then \eqref{eq:fkkpz} can be rewritten as
\begin{equation}
h^\eps(t,x)\stackrel{law}{=} \log \EE_{\eps^{-1}x}\Big[\Psi_{\eps^{-2}t}^{\beta_{\eps}, h_0}\Big].
\label{eq:kpzfk}
\end{equation}

In the following sections, in order to prove Theorem~\ref{thm:main} for $h^\eps(t,x)$, we analyze the limit of $\log \EE_{\eps^{-1}x}\Big[\Psi_{\eps^{-2}t}^{\beta_{\eps}, h_0}\Big]$ as $\eps \to 0$.

The main idea behind the proof for Theorem~\ref{thm:main} is the following decomposition:
\begin{equation}
\log \EE_{\eps^{-1}x}\big[\Psi_{\eps^{-2}t}^{\beta_{\eps}, h_0}\big] = \log \frac{ \EE_{\eps^{-1}x}\big[\Psi_{\eps^{-2}t}^{\beta_{\eps}, h_0}\big]}{\EE_{\eps^{-1}x}\big[\Psi_{\eps^{-2(1-a_\eps)}t}^{\beta_{\eps}, 0}\big]} + \log \EE_{\eps^{-1}x}\big[\Psi_{\eps^{-2(1-a_\eps)}t}^{\beta_{\eps}, 0}\big]
\label{eq:decomp}
\end{equation}
for some $a_{\eps}=o(1)$ satisfying $\eps^{2a_{\eps}}=o(1)$ that we will define later. 

This decomposition is based on the following observation: the random part in the limit of $h^{\eps}(t,x)$ depends only on the white noise $\xi$ in an infinitesimal time window $[t-o(1),t]$ as $\eps \to 0$, and it is independent of the initial condition $h_0$. Thus, we can split $h^{\eps}(t,x)$ into two components: a ``random'' quantity which is the solution to the KPZ equation on time window $[t-o(1),t]$ with zero initial condition, and an ``almost deterministic'' quantity that concentrates to a deterministic value when $\eps \to 0$. The right-hand-side of \eqref{eq:decomp} demonstrates this decomposition, with the first term being the ``almost deterministic'' quantity and the second term being the ``random'' quantity.

The remaining proof consists of two parts. In the first part, we show that the  ``almost deterministic'' quantity converges to the solution of the deterministic KPZ equation with original initial data $h_0$. This is proved in Proposition~\ref{pr:KPZdetermin} below, with the assistance of Proposition~\ref{pr:SHEdetermin}. In the second part, we prove that the local average of the ``random'' quantity converges in law to the Gaussian random variable $\mathcal{N}(0, \sigma_{\gamma}^2) - \frac{1}{2}\log\left(\frac{2\pi}{2\pi - \beta^2}\right)$. This is shown in Proposition~\ref{pr:localave}.

We should note that the idea of this decomposition was previously mentioned in \cite{caravenna2017universality} Remark 2.18 for 2-dimensional SHE with general initial conditions. It has since been utilized in various recent works, including \cite{caravenna2020two, dunlap2022forward, lygkonis2022edwards}, with a particular emphasis in \cite{chatterjee2023weak}.

\subsection*{Acknowledgement}
The author would like to thank their advisor, Yu Gu, for suggesting this problem and providing guidance, and would like to thank Sourav Chatterjee for the suggestion on adding Corollary~\ref{co:localfield}. \added{The author would also like to thank two anonymous referees for multiple constructive suggestions which helped
to improve the presentation. }This work is supported by Yu Gu’s
NSF grant DMS-2203014.

\subsection*{Data Availability} Not applicable.
\subsection*{Declarations} Conflict of interest: The author declares that they have no conflict of interest.

\section{Moment bounds}
In this section, we will discuss some properties of the random variables $\EE_{\eps^{-1}x}\Big[\Psi_{s}^{\beta_{\eps}, 0}\Big]$, $\EE_{0, {\eps}^{-1}x}^{s, {\eps}^{-1}y}\Big[\Psi_{s}^{\beta_{\eps}, 0}\Big]$ and $ \log \EE_{\eps^{-1}x}\Big[\Psi_{s}^{\beta_\eps, 0}\Big]$, where $\Psi_{s}^{\beta_{\eps}, 0}$ is defined as in \eqref{eq:psi} with $f$ being a zero function. These properties will be used in the proof of Theorem~\ref{thm:main} later.

Hereafter, for any $p>0$, we use the notation $\| \cdot \|_{L^p(\Omega)}$ to denote the $L^p(\Omega)$ norm of the probability space $(\Omega, \mathscr{F}, \Prob)$ where the space-time white noise $\xi$ is built on.

\subsection{Second moments}
We will first show that $\EE_{\eps^{-1}x}\Big[\Psi_{s}^{\beta_{\eps}, 0}\Big]$ and $\EE_{0, {\eps}^{-1}x}^{s, {\eps}^{-1}y}\Big[\Psi_{s}^{\beta_{\eps}, 0}\Big]$ have bounded second moments.
Let 
\begin{equation}\label{eq:defofV}
V(x) = \int_{\R^2}\phi(x-y)\phi(y)\dd y.
\end{equation}

We first remark that for any $x, y \in \R^2$,
\begin{equation}
\begin{aligned}
&\E\Big(\EE_{\eps^{-1}x}\Big[\Psi_{s}^{\beta_{\eps}, 0}\Big]\EE_{\eps^{-1}y}\Big[\Psi_{s}^{\beta_{\eps}, 0}\Big]\Big) \\&= \EE_{\eps^{-1}x}\otimes \EE_{\eps^{-1}y} \otimes \E\Big(\exp \big[\beta_{\eps}\int_0^s \int_{\R^2} [\phi(y-B^1_r)+ \phi(y-B^2_r)]\xi(r, y) \dd y \dd r -\beta_{\eps}^2 \|\phi\|_{L^2}^2s\big]\Big)\\
&= \EE_{\eps^{-1}x}\otimes \EE_{\eps^{-1}y}\Big(\exp \big[\frac{\beta_{\eps}^2}{2}\int_0^s \int_{\R^2} [\phi(y-B^1_r)+ \phi(y-B^2_r)]^2 \dd y \dd r -\beta_{\eps}^2 \|\phi\|_{L^2}^2s\big]\Big)\\
&= \EE_{\eps^{-1}x}\otimes \EE_{\eps^{-1}y}\Big(\exp \big[\beta_{\eps}^2 \int_0^s \int_{\R^2} \phi(y-B^1_r)\phi(y-B^2_r) \dd y \dd r \big]\Big)\\
&=  \EE_{\eps^{-1}x}\otimes \EE_{\eps^{-1}y}\Big(\exp \big[\beta_{\eps}^{2} \int_0^s V(B_r^1-B_r^2) \dd r \big]\Big)\\
&=  \EE_{\eps^{-1}\frac{x-y}{\sqrt{2}}}\Big(\exp \big[\beta_{\eps}^2 \int_0^s V(\sqrt{2}B_r) \dd r \big]\Big)\\
&= 1 + \sum_{n=1}^{\infty}\beta_{\eps}^{2n} \int_{0 < s_1 < \dots < s_n<s}\int_{\R^{2n}}\prod_{i=1}^n V(\sqrt{2} x_i)\rho_{s_i - s_{i-1}}(x_{i-1}, x_i) \dd s_1 \dots \dd s_n \dd x_1 \dots \dd x_n
\end{aligned}\label{eq:ucov}
\end{equation}
where we set $x_0 = \eps^{-1}\frac{x-y}{\sqrt{2}}$ \replaced{, }{and }$s_0 = 0$ \replaced{and }{. We} let $\rho_t(x) = (2\pi t)^{-1}e^{-\frac{|x|^2}{2t}}$ be the heat kernel in $d=2$ \replaced{such that }{and }$\rho_t(x,y) := \rho_t(x-y)$. We use $\EE_x\otimes \EE_y$ to denote the expectation in two independent Brownian motions $B^1$ and $B^2$ starting from $x$ and $y$ respectively.

Since
\begin{equation}
\sup_{x\in \R^2} \int_{\R^2} V(\sqrt{2}y)\rho_r(x,y)\dd y \leq \frac{1}{4\pi r} \wedge \|V\|_{\infty},
\end{equation}
where $\|V\|_\infty$ denotes the supremum norm of $V$ on $\R^2$,
\eqref{eq:ucov} is bounded from above by
\[
1+ \sum_{n=1}^{\infty}\beta_{\eps}^{2n} \Big( \int_0^s  \frac{1}{4\pi r} \wedge \|V\|_{\infty}\dd r \Big)^n = 1 + \sum_{n=1}^{\infty}\beta_{\eps}^{2n} \Big( \frac{\log 4\pi \|V\|_{\infty}s}{4 \pi} +\frac{1}{4\pi} \Big)^n.
\]
Now for any fixed $t>0$, we have that as $\eps \to 0$,
\[
\beta_{\eps}^2\Big( \frac{\log 4\pi \|V\|_{\infty}\eps^{-2}t}{4 \pi} +\frac{1}{4\pi}\Big) = \frac{\beta^2}{\log \eps^{-1}} \Big(\frac{\log \eps^{-2}}{4\pi} + \frac{1+\log 4\pi \|V\|_{\infty}t}{4 \pi} \Big) \to \frac{\beta^2}{2\pi}.
\]
Hence when $s \leq \eps^{-2}t$ and $0 < \beta < \sqrt{2\pi}$ , we have
\begin{equation}
\limsup_{\eps \to 0} \E\Big(\EE_{\eps^{-1}x}\Big[\Psi_{s}^{\beta_{\eps}, 0}\Big]\EE_{\eps^{-1}y}\Big[\Psi_{s}^{\beta_{\eps}, 0}\Big]\Big) \leq \sum_{n=0}^{\infty} \Big(\frac{\beta^2}{2\pi}\Big)^n \leq \frac{2\pi}{2\pi-\beta^2}.
\label{eq:general2ndmomentbound}
\end{equation}

We also have a second moment bound for the expectation w.r.t. Brownian bridge.
\begin{lemma}
Let $0 < \beta < \sqrt{2\pi}$, $t>0$ and $x, y \in \R^2$. There exists $C_{\beta} < \infty$ independent of $t, x, y$, such that
\begin{equation}\label{eq:2ndmombb}
\sup_{\eps \leq 1}\sup_{s \leq  \eps^{-2}t} \left\|\EE_{0, {\eps}^{-1}x}^{s, {\eps}^{-1}y}\left[\Psi_s^{\beta_{\eps}, 0}\right]\right\|_{L^2(\Omega)}< C_{\beta}.
\end{equation}
\label{le:2ndmombb}
\end{lemma}
\begin{proof}
By the shear invariance of the environment, for any $\eps >0$ and $x, y \in \R^2$, \[\left\|\EE_{0, {\eps}^{-1}x}^{s, {\eps}^{-1}y}\left[\Psi_s^{\beta_{\eps}, 0}\right]\right\|_{L^2(\Omega)} = \left\| \EE_{0, 0}^{s, 0}\left[\Psi_s^{\beta_{\eps}, 0}\right]\right\|_{L^2(\Omega)}.\] \eqref{eq:2ndmombb} is then derived by applying \cite[(2.5)]{nakajima2023fluctuations}.
\end{proof}

\subsection{Higher moments}
Next we present the boundedness of some higher moments of $\EE_{\eps^{-1}x}\Big[\Psi_{s}^{\beta_{\eps}, 0}\Big]$ and $\EE_{0, {\eps}^{-1}x}^{s, {\eps}^{-1}y}\Big[\Psi_{s}^{\beta_{\eps}, 0}\Big]$. This is proved via hypercontractivity.

\begin{lemma}
Let $0 < \beta < \sqrt{2\pi}$, $t>0$ and $x, y \in \R^2$. There exists some $p_{\beta}>2$ such that $ \forall 2 \leq p < p_{\beta}$, $\exists {C}_{\beta, p} < \infty$ independent of $t, x, y$, such that
\begin{equation}
\sup_{\eps \leq 1}\sup_{s \leq {\eps^{-2}t}}\left\|\EE_{\eps^{-1}x}\left[\Psi_s^{\beta_{\eps}, 0}\right]\right\|_{L^p(\Omega)} \leq {C}_{\beta, p},\label{eq:higherbound1}
\end{equation}
and
\begin{equation}
\sup_{\eps \leq 1}\sup_{s \leq  \eps^{-2}t} \left\|\EE_{0, {\eps}^{-1}x}^{s, {\eps}^{-1}y}\left[\Psi_s^{\beta_{\eps}, 0}\right]\right\|_{L^p(\Omega)} \leq {C}_{\beta, p}.\label{eq:higherbound2}
\end{equation}
\label{le:higher}
\end{lemma}

\begin{proof}
\eqref{eq:higherbound1} is exactly \cite[(5.11)]{caravenna2020two} with $p_{\beta} = 1+2\pi/\beta^2 $.

\eqref{eq:higherbound2} can be proved in a similar way, as $\EE_{0, {\eps}^{-1}x}^{s, {\eps}^{-1}y}\Big[\Psi_s^{\beta_{\eps}, 0}\Big]$ admits a similar Wiener chaos expansion and is bounded in $L^2$. A discrete version of \eqref{eq:higherbound2} for the partition function of point-to-point polymer was proved in \cite[Corollary 2.8 (i)]{gabriel2021central}. 
Here we present a proof in the \replaced{continuum }{continuous }setting.

Let $(\Omega, \mathscr{F}, \Prob)$ be the probability space on which the space-time white noise $\xi$ is built. Let \replaced{$\{T_v, v \geq 0\}$}{$\{T_r, r \geq 0\}$} be the Ornstein-Uhlenbeck semigroup on $L^2(\Omega)$ (see e.g. \cite{nualart2006malliavin} for a reference). Let $\tilde{\xi}$ be an independent copy of $\xi$ built on the probability space $(\tilde{\Omega}, \tilde{\mathscr{F}}, \tilde{\Prob})$. By Mehler’s formula,
\begin{equation}
\begin{aligned}
&T_v \left(\EE_{0, {\eps}^{-1}x}^{s, {\eps}^{-1}y}\left[\Psi_s^{\beta_{\eps}, 0}\right]\right) = T_v \left(\EE_{0, {\eps}^{-1}x}^{s, {\eps}^{-1}y}\left[ \exp \left[\beta_{\eps}\int_0^s \int_{\R^2} \phi(y-B_r)\xi(r, y) \dd y \dd r-\frac{1}{2}\beta_{\eps}^2 \|\phi\|_{L^2}^2s\right]\right]\right) \\
&= \tilde{\E}\EE_{0, {\eps}^{-1}x}^{s, {\eps}^{-1}y}\left[ \exp \left[\beta_{\eps}\int_0^s \int_{\R^2} \phi(y-B_r)\left(e^{-v}\xi(r, y)+\sqrt{1-e^{-2v}}\tilde{\xi}(r,y)\right) \dd y \dd r-\frac{1}{2}\beta_{\eps}^2 \|\phi\|_{L^2}^2s\right]\right],
\end{aligned}\nonumber
\end{equation}
where $\tilde{\E}$ is the expectation w.r.t. $\tilde{\Prob}$. Since \begin{align*}
&\tilde{\E} \left[ \exp \left[\beta_{\eps}\int_0^s \int_{\R^2} \phi(y-B_r)\left(e^{-v}\xi(r, y)+\sqrt{1-e^{-2v}}\tilde{\xi}(r,y)\right) \dd y \dd r-\frac{1}{2}\beta_{\eps}^2 \|\phi\|_{L^2}^2s\right]\right]\\
&= \exp \left[\beta_{\eps}\int_0^s \int_{\R^2} \phi(y-B_r)e^{-v}\xi(r, y)\dd y \dd r-\frac{1}{2}e^{-2v}\beta_{\eps}^2\|\phi\|_{L^2}^2s\right],
\end{align*}
we have 
\begin{equation}
\begin{aligned}
&T_v \left(\EE_{0, {\eps}^{-1}x}^{s, {\eps}^{-1}y}\left[\Psi_s^{\beta_{\eps}, 0}\right]\right) \\
&= \EE_{0, {\eps}^{-1}x}^{s, {\eps}^{-1}y} \left[  \exp \left[\beta_{\eps}\int_0^s \int_{\R^2} \phi(y-B_r)e^{-v}\xi(r, y)\dd y \dd r-\frac{1}{2}e^{-2v}\beta_{\eps}^2\|\phi\|_{L^2}^2s\right] \right]\\
&= \EE_{0, {\eps}^{-1}x}^{s, {\eps}^{-1}y}\left[\Psi_s^{(e^{-v}\beta_{\eps}), 0}\right].
\end{aligned}\label{eq:oubd}
\end{equation}
By the hypercontractivity property of $\{T_v, v \geq 0\}$, for any $v>0$, if $p(v) = e^{2v}+1$, we have
\[
\left\|T_v \left(\EE_{0, {\eps}^{-1}x}^{s, {\eps}^{-1}y}\left[\Psi_s^{\beta_{\eps}, 0}\right]\right)\right\|_{L^{p(v)}(\Omega)} \leq \left\| \EE_{0, {\eps}^{-1}x}^{s, {\eps}^{-1}y}\left[\Psi_s^{\beta_{\eps}, 0}\right] \right\|_{L^2(\Omega)}.
\]
With $0< \beta < \sqrt{2\pi}$, we have $e^{v}\beta < \sqrt{2\pi}$ when $v < \log \frac{\sqrt{2\pi}}{\beta}$. Thus by \eqref{eq:oubd} and Lemma~\ref{le:2ndmombb}, if we restrict $p <  p_\beta = 1+{2\pi}/{\beta^2}$ and let $v=\frac{1}{2}\log(p-1)$ accordingly, then
\[
\sup_{\eps \leq 1}\sup_{s \leq  \eps^{-2}t} \left\| \EE_{0, {\eps}^{-1}x}^{s, {\eps}^{-1}y}\left[\Psi_s^{\beta_{\eps}, 0}\right] \right\|_{L^p(\Omega)} \leq \sup_{\eps \leq 1}\sup_{s \leq  \eps^{-2}t} \left\| \EE_{0, {\eps}^{-1}x}^{s, {\eps}^{-1}y}\left[\Psi_s^{(e^v\beta_{\eps}), 0}\right] \right\|_{L^2(\Omega)} < C_{\beta, v},
\]
for some constant $C_{\beta, v}$.
\end{proof}

\subsection{Negative moments}
We next bound the negative moments of $\EE_{\eps^{-1}x}\Big[\Psi_{s}^{\beta_{\eps}, 0}\Big]$. The following lemma is proved in  \cite{caravenna2020two} via concentration inequality. 

\begin{lemma}[\cite{caravenna2020two}, (5.13)]\label{le:negmom}
Let $0 < \beta < \sqrt{2\pi}$, $t>0$ and $x \in \R^2$. For any $p >0$, there exists ${C}_{p, \beta} < \infty$ independent of $t$ and $x$, such that
\[
\sup_{\eps \leq 1}\sup_{s \leq  \eps^{-2}t} \E\left[\Big(\EE_{\eps^{-1}x}\Big[\Psi_{s}^{\beta_{\eps}, 0}\Big]\Big)^{-p}\right] \leq {C}_{p, \beta}.
\]
\end{lemma}

A direct corollary is the following moment bounds of the logarithm.
\begin{corollary}\label{co:logmom}
Let $0 < \beta < \sqrt{2\pi}$, $t>0$ and $x \in \R^2$. For any $p >0$, there exists ${C}_{p, \beta} < \infty$ independent of $t$ and $x$, such that 
\[
\sup_{\eps \leq 1} \sup_{s \leq \eps^{-2}t}\E\Big(\left|\log \EE_{\eps^{-1}x}\Big[\Psi_{s}^{\beta_\eps, 0}\Big]\right|^{p}\Big) \leq  {C}_{p, \beta}.
\]
\end{corollary}
\begin{proof}
For any $p >0$, there exists $C_p>0$ such that $|x|^p \leq C_p (e^x + e^{-x})$ for any $x\in \R$. Apply \eqref{eq:general2ndmomentbound} and Lemma~\ref{le:negmom}.
\end{proof}

If  $h_0:\R^2 \to \R$ is bounded, \eqref{eq:general2ndmomentbound} and Lemma~\ref{le:negmom} remain valid for $\Psi_s^{\beta_\eps, h_0}$ after we multiply some constants to the bounds.  As a result, Corollary~\ref{co:logmom} also holds for $\Psi_s^{\beta_\eps, h_0}$.

\section{Proof of Theorem~\ref{thm:main}}
The proof here is inspired by \cite{chatterjee2023weak}. 

We first state the following propositions that we will prove later.

\begin{proposition}\label{pr:SHEdetermin}
Let $0<\beta<\sqrt{2\pi}$ and $h_0 \in C(\R^2)$ be bounded and Lipschitz continuous. Let $a_{\eps} = o(1)$. There exist $C_{\beta}>0$ and $\eps_0>0$, such that when $0 < \eps < \eps_0$, for any $t>0$, $x \in \R^2$, we have 
\begin{equation}
\E\left[\left|\frac{ \EE_{\eps^{-1}x}\big[\Psi_{\eps^{-2}t}^{\beta_{\eps}, h_0}\big]}{\EE_{\eps^{-1}x}\big[\Psi_{\eps^{-2(1-a_\eps)}t}^{\beta_{\eps}, 0}\big]}  - \exp[\bar{h}(t - \eps^{2a_{\eps}}t,x)]\right|^2\right] \leq C_{\beta}(\eps^{2a_{\eps}}t+a_{\eps}^{1/3}),
\label{eq:conv2SHE}
\end{equation}
where $\bar{h}(t,x)$ is defined as in \eqref{eq:deterkpz}.
\end{proposition}

\begin{proposition}\label{pr:KPZdetermin}
With the same assumptions as in Proposition~\ref{pr:SHEdetermin}, there exist $C_{\beta}>0$ and $\eps_0>0$, such that when $0 < \eps < \eps_0$, for any $t>0$, $x \in \R^2$, we have 
\begin{equation}
\E\left[\left| \log \frac{ \EE_{\eps^{-1}x}\big[\Psi_{\eps^{-2}t}^{\beta_{\eps}, h_0}\big]}{\EE_{\eps^{-1}x}\big[\Psi_{\eps^{-2(1-a_\eps)}t}^{\beta_{\eps}, 0}\big]}  - \bar{h}(t - \eps^{2a_{\eps}}t,x)\right|^2\right] \leq C_{\beta} (\eps^{2a_{\eps}}t+a_{\eps}^{1/3})^{1/8}.
\label{eq:conv2KPZ}
\end{equation}
\end{proposition}

\begin{proposition}\label{pr:localave}
Let $0<\beta<\sqrt{2\pi}$, $r_{\eps} = \eps^{1-\gamma}$ for some $0 \leq \gamma \leq 1$ and $a_{\eps} = (\log \eps^{-1})^{-1/2}$. For any $t>0$, $x \in \R^2$, as $\eps \to 0$, 
\begin{equation}
\frac{1}{|B(x, r_{\eps})|}\int_{|y-x| \leq r_{\eps}} \log \EE_{\eps^{-1}y}\big[\Psi_{\eps^{-2(1-a_\eps)}t}^{\beta_{\eps}, 0}\big] \dd y \stackrel{d}{\to} \mathcal{N}(0, \sigma_{\gamma}^2) - \frac{1}{2}\log\left(\frac{2\pi}{2\pi - \beta^2}\right),
\label{eq:localave}\end{equation}
where $\mathcal{N}(0, \sigma_{\gamma}^2) $ is a normal distribution with mean 0 and variance $\sigma_{\gamma}^2 := \log\left( \frac{2\pi - \beta^2\gamma}{2\pi - \beta^2}\right)$.  When $0 \leq \gamma <1$, $\sigma^2_{\gamma} >0$. When $\gamma =1$, $\sigma^2_{\gamma} =0$.
\end{proposition}

We first use Proposition~\ref{pr:KPZdetermin} and Proposition~\ref{pr:localave} to prove Theorem~\ref{thm:main}.

Let \begin{equation}\label{eq:aepsilon}
a_{\eps} = (\log \eps^{-1})^{-1/2}.
\end{equation} 
Then we have  $a_{\eps}  = o(1)$ and $\eps^{a_\eps} = o(1)$. In particular, as $\eps \to 0$, 
\begin{equation}\label{eq:limitaeps}
C_{\beta}(\eps^{2a_{\eps}}t+a_{\eps}^{1/3})^{1/8}\to 0.
\end{equation}
By applying \eqref{eq:fkkpz} and the decomposition in \eqref{eq:decomp}, we have
\begin{equation*}
\begin{aligned}
& \frac{1}{|B(x, r_{\eps})|} \int_{|y - x|\leq r_{\eps}} h^{\eps}(t,y) \dd y  \stackrel{law}{=}  \frac{1}{|B(x, r_{\eps})|} \int_{|y - x|\leq r_{\eps}} \log \EE_{\eps^{-1}y}\big[\Psi_{\eps^{-2}t}^{\beta_{\eps}, h_0}\big] \dd y
\\&=  \frac{1}{|B(x, r_{\eps})|} \int_{|y - x|\leq r_{\eps}}  \log \frac{ \EE_{\eps^{-1}y}\big[\Psi_{\eps^{-2}t}^{\beta_{\eps}, h_0}\big]}{\EE_{\eps^{-1}y}\big[\Psi_{\eps^{-2(1-a_\eps)}t}^{\beta_{\eps}, 0}\big]} + \log \EE_{\eps^{-1}y}\big[\Psi_{\eps^{-2(1-a_\eps)}t}^{\beta_{\eps}, 0}\big] \dd y\\
& = \mathcal{I}_1 +  \mathcal{I}_2 +  \mathcal{I}_3,
\end{aligned}
\label{eq:split}
\end{equation*}
where
\begin{align*}
\mathcal{I}_1 &:= \frac{1}{|B(x, r_{\eps})|} \int_{|y - x|\leq r_{\eps}}  \log \frac{ \EE_{\eps^{-1}y}\big[\Psi_{\eps^{-2}t}^{\beta_{\eps}, h_0}\big]}{\EE_{\eps^{-1}y}\big[\Psi_{\eps^{-2(1-a_\eps)}t}^{\beta_{\eps}, 0}\big]} \dd y- \frac{1}{|B(x, r_{\eps})|} \int_{|y - x|\leq r_{\eps}}  \bar{h}(t-\eps^{2a_\eps}t, y) \dd y,\\
\mathcal{I}_2 &:=  \frac{1}{|B(x, r_{\eps})|} \int_{|y - x|\leq r_{\eps}}  \log \EE_{\eps^{-1}y}\big[\Psi_{\eps^{-2(1-a_\eps)}t}^{\beta_{\eps}, 0}\big] \dd y,\\
\mathcal{I}_3 &:= \frac{1}{|B(x, r_{\eps})|} \int_{|y - x|\leq r_{\eps}}  \bar{h}(t-\eps^{2a_\eps}t, y) \dd y.
\end{align*}

By the \replaced{Minkowski’s integral inequality}{Minkowski inequality}, Proposition~\ref{pr:KPZdetermin} and \eqref{eq:limitaeps}, when $\eps < \eps_0$,
\begin{equation*}
\begin{aligned}
\|\mathcal{I}_1\|_{L^2(\Omega)} 
& = \left\|\frac{1}{|B(x, r_{\eps})|} \int_{|y - x|\leq r_{\eps}}  \log \frac{ \EE_{\eps^{-1}y}\big[\Psi_{\eps^{-2}t}^{\beta_{\eps}, h_0}\big]}{\EE_{\eps^{-1}y}\big[\Psi_{\eps^{-2(1-a_\eps)}t}^{\beta_{\eps}, 0}\big]} \dd y - \frac{1}{|B(x, r_{\eps})|} \int_{|y - x|\leq r_{\eps}}  \bar{h}(t-\eps^{2a_\eps}t, y) \dd y \right\|_{L^2(\Omega)}
 \\&\leq  \frac{1}{|B(x, r_{\eps})|} \int_{|y - x|\leq r_{\eps}}\left\| \log \frac{ \EE_{\eps^{-1}y}\big[\Psi_{\eps^{-2}t}^{\beta_{\eps}, h_0}\big]}{\EE_{\eps^{-1}y}\big[\Psi_{\eps^{-2(1-a_\eps)}t}^{\beta_{\eps}, 0}\big]} - \bar{h}(t-\eps^{2a_\eps}t, y) \right\|_{L^2(\Omega)} \dd y\\
& \leq  \frac{1}{|B(x, r_{\eps})|} \int_{|y - x|\leq r_{\eps}} C_{\beta}^{1/2} (\eps^{2a_{\eps}}t+a_{\eps}^{1/3})^{1/16} \dd y\\
& = C_{\beta}^{1/2} (\eps^{2a_{\eps}}t+a_{\eps}^{1/3})^{1/16} \to 0, \quad \text{as } \eps \to 0.
\end{aligned}
\end{equation*}
In particular, $\mathcal{I}_1$ converges to 0 in distribution.

By Proposition~\ref{pr:localave}, $\mathcal{I}_2$ converges to $ \mathcal{N}(0, \sigma_{\gamma}^2) - \frac{1}{2}\log\left(\frac{2\pi}{2\pi - \beta^2}\right)$ in distribution for all $0 \leq \gamma \leq 1$.

The last term $\mathcal{I}_3$ is deterministic. By Lebesgue's dominated convergence theorem, if $0 \leq \gamma < 1$, $\mathcal{I}_3 \to \bar{h}(t,x)$ as $\eps \to 0$. If $\gamma = 1$,
\[\mathcal{I}_3 \to \frac{1}{|B(x,1)|} \int_{|y-x|\leq 1} \bar{h}(t,y) \dd y, \quad \text{as } \eps \to 0.\]

Therefore, by proving Proposition~\ref{pr:KPZdetermin} and Proposition~\ref{pr:localave}, we would have proved Theorem~\ref{thm:main}.\deleted{ The rest of the paper is to prove these two propositions. 
}
In order to prove Proposition~\ref{pr:KPZdetermin}, we shall first prove Proposition~\ref{pr:SHEdetermin}.

\subsection{Proof of Proposition \ref{pr:SHEdetermin}}
Hereafter, we always assume $0<\beta<\sqrt{2\pi}$. By the Markov property, for any $0 <s < \eps^{-2}t$,
\begin{equation*}
\begin{aligned}
&\EE_{\eps^{-1}x}\big[\Psi_{\eps^{-2}t}^{\beta_{\eps}, h_0}\big] \\& = \EE_{\eps^{-1}x}\Big[\exp \Big\{h_0(\eps B_{\eps^{-2}t})+\beta_\eps \int_0^{\eps^{-2}t} \int_{\R^2} \phi(B_r-y)\xi(r, y) \dd y \dd r  -\frac{1}{2}\beta_\eps^2 \eps^{-2}t \|\phi\|_{L^2(\R^2)}^2\Big\}\Big]\\
&= \EE_{\eps^{-1}x}\Big[\exp \Big\{h_0(\eps B_{\eps^{-2}t})+\beta_\eps\int_s^{\eps^{-2}t} \int_{\R^2} \phi(B_r-y)\xi(r, y) \dd y \dd r \\
&\quad +\beta_\eps\int_{0}^{s} \int_{\R^2} \phi(B_r-y)\xi(r, y) \dd y \dd r -\frac{1}{2}\beta_\eps^2 \eps^{-2}t \|\phi\|_{L^2(\R^2)}^2\Big\}\Big]\\
&= \int_{z \in \R^2} 
\rho_{s}(\eps^{-1}x,z) \EE_{0, \eps^{-1}x}^{s, z}\Big[\exp \Big\{\beta_\eps\int_{0}^{s} \int_{\R^2} \phi(B_r^B-y)\xi(r, y) \dd y \dd r - \frac{1}{2}\beta_\eps^2 s \|\phi\|_{L^2(\R^2)}^2\Big\}\Big]\\
 &\quad \EE_{z}\Big[\exp \Big\{h_0(\eps \tilde{B}_{\eps^{-2}t-s})+\beta_\eps\int_0^{\eps^{-2}t-s} \int_{\R^2} \phi(\tilde{B}_r-y)\xi(r+s, y) \dd y \dd r -\frac{1}{2}\beta_\eps^2 (\eps^{-2}t-s) \|\phi\|_{L^2(\R^2)}^2\Big\}\Big] \dd z\\
&=  \int_{z \in \R^2} 
\rho_{s}(\eps^{-1}x,z) \EE_{0, \eps^{-1}x}^{s, z}\Big[\Psi_{s}^{\beta_{\eps}, 0} (B^B, \xi) \Big]\EE_z\Big[\Psi_{\eps^{-2}t-s}^{\beta_{\eps}, h_0} (\tilde{B}, \xi') \Big] \dd z.
\end{aligned}
\end{equation*}
Here $\xi'$ is a time shift of $\xi$ by $s$, i.e. $\xi'(r, \cdot) = \xi(s+r, \cdot)$ \replaced{,}{.} $\{B^B_r: 0 \leq r \leq s\}$ is a Brownian bridge starting from $ \eps^{-1}x$ at time 0 and ending at $z$ at time $s$\replaced{, and}{ .} $\{\tilde{B}_r: 0 \leq r \leq  \eps^{-2}t-s\}$ is a Brownian motion starting from $z$\replaced{, where}{.} $B^B$ and $\tilde{B}$ are independent.

For any $s>0$,
\begin{equation}\label{eq:denominatorterm}
\EE_{\eps^{-1}x}\big[\Psi_{s}^{\beta_{\eps}, 0}\big] = \int_{z \in \R^2} \rho_s(\eps^{-1}x,z)\EE_{0, \eps^{-1}x}^{s, z}\Big[\Psi_{s}^{\beta_{\eps}, 0} (B^B, \xi) \Big]\dd z.
\end{equation}
\replaced{In the rest of the paper, we will always set }{Let }\[s = \eps^{-2(1-a_\eps)}t,\]
where $a_\eps$ is defined as in \eqref{eq:aepsilon} and $\eps^2s=\eps^{2a_\eps}=o(1)$. By \eqref{eq:denominatorterm}, the ratio expression in \eqref{eq:conv2SHE} satisfies that
\begin{equation}
 \frac{ \EE_{\eps^{-1}x}\big[\Psi_{\eps^{-2}t}^{\beta_{\eps}, h_0}\big]}{\EE_{\eps^{-1}x}\big[\Psi_{\eps^{-2(1-a_\eps)}t}^{\beta_{\eps}, 0}\big]} =  \frac{\int_{z \in \R^2} 
\rho_{s}(\eps^{-1}x,z) \EE_{0, \eps^{-1}x}^{s, z}\Big[\Psi_{s}^{\beta_{\eps}, 0} (B^B, \xi) \Big]\EE_z\Big[\Psi_{\eps^{-2}t-s}^{\beta_{\eps}, h_0} (\tilde{B}, \xi') \Big] \dd z}{\int_{z \in \R^2} \rho_s(\eps^{-1}x,z)\EE_{0, \eps^{-1}x}^{s, z}\Big[\Psi_{s}^{\beta_{\eps}, 0} (B^B, \xi) \Big]\dd z} \label{eq:averagedpol}.
\end{equation}
We could thus interprete the ratio $ \frac{ \EE_{\eps^{-1}x}\big[\Psi_{\eps^{-2}t}^{\beta_{\eps}, h_0}\big]}{\EE_{\eps^{-1}x}\big[\Psi_{\eps^{-2(1-a_\eps)}t}^{\beta_{\eps}, 0}\big]} $ as a \added{(randomly)} weighted average of $\EE_z\Big[\Psi_{\eps^{-2}t-s}^{\beta_{\eps}, h_0}  (\tilde{B}, \xi') \Big] $ over $z \in \R^2$.

\deleted{To simplify the notations in the computation, we set $s:= \eps^{-2(1-a_\eps)}t$ hereafter. }
Let $\mathscr{F}_{s}$ be the $\sigma$-algebra generated by the environment $\{\xi(r, \cdot): 0 \leq r < s \}$.  Let \replaced{$\EEE^{\mathscr{F}_{s}}$, $\mathbf{{P}}^{\mathscr{F}_{s}}$, $\mathbf{{V}ar}^{\mathscr{F}_{s}}$, $\mathbf{{C}ov}^{\mathscr{F}_{s}}$ }{$\E'$, $\Prob'$, $\Var'$, $\Cov'$} denote the conditional expectation, conditional probability, conditional variance and conditional covariance given $\mathscr{F}_{s}$, respectively. In particular, the noise (with a time shift by $s$) $\xi'$ is independent of the filtration $\mathscr{F}_{s}$.

We first show that the conditional expectation of the ratio $ \frac{ \EE_{\eps^{-1}x}\big[\Psi_{\eps^{-2}t}^{\beta_{\eps}, h_0}\big]}{\EE_{\eps^{-1}x}\big[\Psi_{\eps^{-2(1-a_\eps)}t}^{\beta_{\eps}, 0}\big]} $ in \eqref{eq:averagedpol} w.r.t. $\mathscr{F}_{s}$ is approximately deterministic as $\eps \to 0$.
\begin{lemma}\label{le:condmean}
Let $\bar{h}(t,x)$ be as in \eqref{eq:deterkpz} and $a_{\eps}<1$. There exists some $C_{\beta}>0$ such that for any $t>0$, $x \in \R^2$,  
\begin{equation}
\left\|\mathbf{{E}}^{\mathscr{F}_{s}} \left\{  \frac{ \EE_{\eps^{-1}x}\big[\Psi_{\eps^{-2}t}^{\beta_{\eps}, h_0}\big]}{\EE_{\eps^{-1}x}\big[\Psi_{\eps^{-2(1-a_\eps)}t}^{\beta_{\eps}, 0}\big]} \right\} - \exp\big[\bar{h}(t- \eps^{2a_{\eps}}t,x)\big]\right\|_{L^2(\Omega)} \leq C_{\beta} \eps^{a_{\eps}}\sqrt{t}.\label{eq:condmean}
\end{equation}
\end{lemma}

\added{We leave the proof of Lemma~\ref{le:condmean} to Appendix~\ref{ap:meanlemma}.}

We next show that the expectation of the conditional variance \[\E \left[\mathbf{{V}ar}^{\mathscr{F}_{s}} \left\{  \frac{ \EE_{\eps^{-1}x}\big[\Psi_{\eps^{-2}t}^{\beta_{\eps}, h_0}\big]}{\EE_{\eps^{-1}x}\big[\Psi_{\eps^{-2(1-a_\eps)}t}^{\beta_{\eps}, 0}\big]} \right\}\right]\] converges to 0 as $\eps \to 0$. \added{This result is basically saying that the randomness from the shifted white noise $\xi'$ is not contributing to the randomness in the ratio $ \frac{ \EE_{\eps^{-1}x}\big[\Psi_{\eps^{-2}t}^{\beta_{\eps}, h_0}\big]}{\EE_{\eps^{-1}x}\big[\Psi_{\eps^{-2(1-a_\eps)}t}^{\beta_{\eps}, 0}\big]} $ as $\eps \to 0$.}  In fact, we have the following bound and we leave its proof to Appendix~\ref{ap:varlemma}.
\begin{lemma}\label{le:condvar}
Let $a_{\eps}= o(1)$. There exists some $C_{\beta} >0$ and $\eps_0 >0$ such that for any $x \in \R^2$ and $t >0$, when $0 < \eps < \eps_0$,
\[
\E \left[\mathbf{{V}ar}^{\mathscr{F}_{s}} \left\{  \frac{ \EE_{\eps^{-1}x}\big[\Psi_{\eps^{-2}t}^{\beta_{\eps}, h_0}\big]}{\EE_{\eps^{-1}x}\big[\Psi_{\eps^{-2(1-a_\eps)}t}^{\beta_{\eps}, 0}\big]} \right\}\right] \leq C_{\beta}  a_{\eps}^{1/3}.\]
\end{lemma}

\replaced{Now by the law of total expectation, we have }{Since}
\begin{equation*}
\begin{aligned}
&\E\left[\frac{ \EE_{\eps^{-1}x}\big[\Psi_{\eps^{-2}t}^{\beta_{\eps}, h_0}\big]}{\EE_{\eps^{-1}x}\big[\Psi_{\eps^{-2(1-a_\eps)}t}^{\beta_{\eps}, 0}\big]}  - \exp\left[\bar{h}(t- \eps^{2a_{\eps}}t,x)\right]\right]^2 \\
& = \E \left[\left(\mathbf{{E}}^{\mathscr{F}_{s}} \left\{  \frac{ \EE_{\eps^{-1}x}\big[\Psi_{\eps^{-2}t}^{\beta_{\eps}, h_0}\big]}{\EE_{\eps^{-1}x}\big[\Psi_{\eps^{-2(1-a_\eps)}t}^{\beta_{\eps}, 0}\big]} \right\} - \exp\big[\bar{h}(t- \eps^{2a_{\eps}}t,x)\big]\right)^2\right] + \E \left[\mathbf{{V}ar}^{\mathscr{F}_{s}} \left\{  \frac{ \EE_{\eps^{-1}x}\big[\Psi_{\eps^{-2}t}^{\beta_{\eps}, h_0}\big]}{\EE_{\eps^{-1}x}\big[\Psi_{\eps^{-2(1-a_\eps)}t}^{\beta_{\eps}, 0}\big]} \right\}\right],
\end{aligned}
\end{equation*}
which makes \eqref{eq:conv2SHE} a direct result of Lemma~\ref{le:condmean} and Lemma~\ref{le:condvar}.

\subsection{Proof of Proposition \ref{pr:KPZdetermin}}
We then use Proposition \ref{pr:SHEdetermin} to prove Proposition \ref{pr:KPZdetermin}. We first start with the $L^{1/2}(\Omega)$ norm.

\begin{lemma}
Under the same assumptions as in Proposition \ref{pr:KPZdetermin}, there exist some $C_{\beta}>0$ and $\eps_0>0$, such that when $0 < \eps < \eps_0$, we have 
\begin{equation}
\E\left[\Big| \log \frac{ \EE_{\eps^{-1}x}\big[\Psi_{\eps^{-2}t}^{\beta_{\eps}, h_0}\big]}{\EE_{\eps^{-1}x}\big[\Psi_{\eps^{-2(1-a_\eps)}t}^{\beta_{\eps}, 0}\big]}  - \bar{h}(t- \eps^{2a_{\eps}}t,x)\Big|^{1/2}\right] \leq C_{\beta} (\eps^{2a_{\eps}}t+a_{\eps}^{1/3})^{1/4}.
\label{eq:conv2KPZhalf}\nonumber
\end{equation}
\label{le:cov2kpzhalf}
\end{lemma}

\begin{proof}
For any $a, b >0$, we have  $|\log a - \log b|^{1/2} \leq \sqrt{|1-\frac{a}{b}|} +  \sqrt{|1-\frac{b}{a}|}=\sqrt{|a-b|}(a^{-1/2}+ b^{-1/2})$. Thus
\begin{equation*}
\begin{aligned}
&\E\left[\left| \log \frac{ \EE_{\eps^{-1}x}\big[\Psi_{\eps^{-2}t}^{\beta_{\eps}, h_0}\big]}{\EE_{\eps^{-1}x}\big[\Psi_{s}^{\beta_{\eps}, 0}\big]}  - \bar{h}(t- \eps^{2a_{\eps}}t,x)\right|^{\frac{1}{2}}\right]\\ 
& \leq \E\left\{\left|  \frac{ \EE_{\eps^{-1}x}\big[\Psi_{\eps^{-2}t}^{\beta_{\eps}, h_0}\big]}{\EE_{\eps^{-1}x}\big[\Psi_{s}^{\beta_{\eps}, 0}\big]}  -\exp [\bar{h}(t- \eps^{2a_{\eps}}t,x)] \right|^{\frac{1}{2}} \left[\left( \frac{ \EE_{\eps^{-1}x}\big[\Psi_{\eps^{-2}t}^{\beta_{\eps}, h_0}\big]}{\EE_{\eps^{-1}x}\big[\Psi_{s}^{\beta_{\eps}, 0}\big]}  \right)^{-\frac{1}{2}}+\exp [-\frac{1}{2}\bar{h}(t- \eps^{2a_{\eps}}t,x)] \right]\right\} \\
& \leq  \sqrt{2}\left[\E\left|  \frac{ \EE_{\eps^{-1}x}\big[\Psi_{\eps^{-2}t}^{\beta_{\eps}, h_0}\big]}{\EE_{\eps^{-1}x}\big[\Psi_{s}^{\beta_{\eps}, 0}\big]}  -\exp [\bar{h}(t- \eps^{2a_{\eps}}t,x)] \right|\right]^{\frac{1}{2}}  \left[ \E \left( \frac{ \EE_{\eps^{-1}x}\big[\Psi_{\eps^{-2}t}^{\beta_{\eps}, h_0}\big]}{\EE_{\eps^{-1}x}\left[\Psi_{s}^{\beta_{\eps}, 0}\right]}  \right)^{-1} +\exp [-\bar{h}(t- \eps^{2a_{\eps}}t,x)]\right]^{\frac{1}{2}}.
\end{aligned}
\end{equation*}
Note that 
\begin{equation}
\E \left( \frac{ \EE_{\eps^{-1}x}\big[\Psi_{\eps^{-2}t}^{\beta_{\eps}, h_0}\big]}{\EE_{\eps^{-1}x}\big[\Psi_{s}^{\beta_{\eps}, 0}\big]}  \right)^{-1} \leq \left\|\EE_{\eps^{-1}x}\big[\Psi_{s}^{\beta_{\eps}, 0}\big] \right\|_2 \left\|\frac{1}{\EE_{\eps^{-1}x}\big[\Psi_{\eps^{-2}t}^{\beta_{\eps}, h_0}\big]}\right\|_2.\label{eq:l1halfnorm}
\end{equation}

By \eqref{eq:general2ndmomentbound} and Lemma ~\ref{le:negmom}, together with the assumption that $h_0$ is bounded, \eqref{eq:l1halfnorm} is bounded by some constant $C_{\beta}$. Lemma~\ref{le:cov2kpzhalf} is then the consequence of the Cauchy-Schwarz inequality and Proposition~\ref{pr:SHEdetermin}.
\end{proof}

Now to prove  Proposition \ref{pr:KPZdetermin},  \replaced{we apply the Cauchy-Schwarz
inequality to the following expression:}{by the Cauchy-Schwarz inequality, we have}
\begin{equation*}
\begin{aligned}
& \E\left[\Big| \log \frac{ \EE_{\eps^{-1}x}\big[\Psi_{\eps^{-2}t}^{\beta_{\eps}, h_0}\big]}{\EE_{\eps^{-1}x}\big[\Psi_{s}^{\beta_{\eps}, 0}\big]}  - \bar{h}(t- \eps^{2a_{\eps}}t,x)\Big|^{2}\right] \\
& = \E\left[\left| \log \frac{ \EE_{\eps^{-1}x}\big[\Psi_{\eps^{-2}t}^{\beta_{\eps}, h_0}\big]}{\EE_{\eps^{-1}x}\big[\Psi_{s}^{\beta_{\eps}, 0}\big]}  - \bar{h}(t- \eps^{2a_{\eps}}t,x)\right|^{1/4}\left| \log \frac{ \EE_{\eps^{-1}x}\big[\Psi_{\eps^{-2}t}^{\beta_{\eps}, h_0}\big]}{\EE_{\eps^{-1}x}\big[\Psi_{s}^{\beta_{\eps}, 0}\big]}  - \bar{h}(t- \eps^{2a_{\eps}}t,x)\right|^{7/4}\right],
\end{aligned}
\end{equation*}
\replaced{and bound the ${L^{7/2}(\Omega)}$ term by using}{By the discussion after} Corollary~\ref{co:logmom} and the discussion after it\replaced{. Then there}{, there} exists some constant $C_{\beta}$ such that 
the above is bounded by 
\[C_{\beta} \left\| \log \frac{ \EE_{\eps^{-1}x}\big[\Psi_{\eps^{-2}t}^{\beta_{\eps}, h_0}\big]}{\EE_{\eps^{-1}x}\big[\Psi_{\eps^{-2(1-a_\eps)}t}^{\beta_{\eps}, 0}\big]}  - \bar{h}(t- \eps^{2a_{\eps}}t,x)\right\|_{L^{1/2}(\Omega)}^{1/4}.\]
We now apply Lemma~\ref{le:cov2kpzhalf}.

\subsection{Proof of Proposition~\ref{pr:localave}}
We shall prove that  \[\frac{1}{|B(x, r_{\eps})|}\int_{|y-x| \leq r_{\eps}} \log \EE_{\eps^{-1}y}\big[\Psi_{\eps^{-2(1-a_\eps)}t}^{\beta_{\eps}, 0}\big] \dd y\] converges to the Gaussian random variable $ \mathcal{N}(0, \sigma_{\gamma}^2) - \frac{1}{2}\log\left(\frac{2\pi}{2\pi - \beta^2}\right)$ as $\eps \to 0$. To do so, we prove the convergence of all its moments.

By the shear invariance of the space-time white noise, we can set $x=0$ without loss of generality.

We first show the convergence of the mean. 

By Corollary~\ref{co:logmom}, for any $t>0$, $y \in \R^2$, $\log \EE_{\eps^{-1}y}\left[\Psi_{s}^{\beta_{\eps}, 0}\right]$
is uniformly integrable. With $a_\eps = o(1)$, by Theorem~\ref{thm:kpzpw}, its expectation converges to $ - \frac{1}{2}\log\left(\frac{2\pi}{2\pi - \beta^2}\right)$ as $\eps \to 0$. 
By a change of variable $\frac{y}{r_\eps} \to \tilde{y}$,
\[\E \left(\frac{1}{|B(0, r_{\eps})|}\int_{|y| \leq r_{\eps}} \log \EE_{\eps^{-1}y}\big[\Psi_{s}^{\beta_{\eps}, 0}\big] \dd y\right) = \E \left(\frac{1}{|B(0, 1)|}\int_{|\tilde{y}| \leq 1} \log \EE_{\eps^{-1}{\tilde{y}r_\eps}}\big[\Psi_{s}^{\beta_{\eps}, 0}\big] \dd \tilde{y}\right).
\]
By Lebesgue's dominated convergence theorem, 
\[\E \left(\frac{1}{|B(0, 1)|}\int_{|\tilde{y}| \leq 1} \log \EE_{\eps^{-1}{\tilde{y}r_\eps}}\big[\Psi_{s}^{\beta_{\eps}, 0}\big] \dd \tilde{y}\right)\to  - \frac{1}{2}\log\left(\frac{2\pi}{2\pi - \beta^2}\right), \quad \text{ as } \eps \to 0.\]

For the second moment, with  $\frac{y}{r_\eps} \to \tilde{y}$ and $\frac{y'}{r_\eps} \to \tilde{y}'$,
\begin{equation*}
\begin{aligned}
& \Var \left( \frac{1}{|B(0, r_{\eps})|}\int_{|y| \leq r_{\eps}} \log \EE_{\eps^{-1}y}\big[\Psi_{s}^{\beta_{\eps}, 0}\big] \dd y \right) \\
&=  \frac{1}{|B(0, r_{\eps})|^2}\int_{|y| \leq r_{\eps}} \int_{|y'| \leq r_{\eps}}\Cov \left( \log \EE_{\eps^{-1}y}\big[\Psi_{s}^{\beta_{\eps}, 0}\big],  \log \EE_{\eps^{-1}y'}\big[\Psi_{s}^{\beta_{\eps}, 0}\big]
 \right) \dd y \dd y'\\
 & = \frac{1}{|B(0, 1)|^2}\int_{|\tilde{y}| \leq 1} \int_{|\tilde{y}'| \leq 1} \Cov \left( \log \EE_{\eps^{-1}{\tilde{y}r_\eps}}\big[\Psi_{s}^{\beta_{\eps}, 0}\big],  \log \EE_{\eps^{-1}{\tilde{y}'r_\eps}}\big[\Psi_{s}^{\beta_{\eps}, 0}\big]
 \right) \dd \tilde{y} \dd \tilde{y}' .
\end{aligned}
\end{equation*}
Again, Corollary~\ref{co:logmom} gives the uniform integrability. By Theorem~\ref{thm:kpzpw} and the continuous mapping theorem, for any $\tilde{y}, \tilde{y}' \in \R^2$, as $\eps \to 0$,
 \[\Cov \left( \log \EE_{\eps^{-1}{\tilde{y}r_\eps}}\big[\Psi_{s}^{\beta_{\eps}, 0}\big],  \log \EE_{\eps^{-1}{\tilde{y}'r_\eps}}\big[\Psi_{s}^{\beta_{\eps}, 0}\big]
 \right) \to \Cov[Y_1, Y_2],\]
where $(Y_1, Y_2)$ are the jointly Gaussian random variables defined in Theorem~\ref{thm:kpzpw} with $\zeta_{1,2} = \gamma$. In particular, since $\Cov[Y_1, Y_2]  = \log\left( \frac{2\pi - \beta^2\gamma}{2\pi - \beta^2}\right)=: \sigma_{\gamma}^2$, we have \[\Var \left( \frac{1}{|B(0, r_{\eps})|}\int_{|y| \leq r_{\eps}} \log \EE_{\eps^{-1}y}\big[\Psi_{s}^{\beta_{\eps}, 0}\big] \dd y \right) \to  \sigma_{\gamma}^2, \quad \text{ as } \eps \to 0.\]

The convergence of any higher moments can be proved following a same procedure. In fact, for any $p \geq 1$, with $\frac{y_i}{r_\eps} \to \tilde{y_i}$ for all $1 \leq i \leq p$,
\begin{equation}
\begin{aligned}
&\E\left(\left[ \frac{1}{|B(0, r_{\eps})|}\int_{|y| \leq r_{\eps}} \log \EE_{\eps^{-1}y}\left[\Psi_{s}^{\beta_{\eps}, 0}\right] \dd y - \E \left( \frac{1}{|B(0, r_{\eps})|}\int_{|y| \leq r_{\eps}} \log \EE_{\eps^{-1}y}\left[\Psi_{s}^{\beta_{\eps}, 0}\right] \dd y \right) \right]^p \right)\\
&= \frac{1}{|B(0, r_{\eps})|^p}\int_{|y_1| \leq r_{\eps}}\dots\int_{|y_p| \leq r_{\eps}} \E \left[ \prod_{i=1}^p  \left( \log \EE_{\eps^{-1}y_i}\big[\Psi_{s}^{\beta_{\eps}, 0}\big] - \E  \log \EE_{\eps^{-1}y_i}\big[\Psi_{s}^{\beta_{\eps}, 0}\big] \right)\right] \dd y_1 \dots \dd y_p\\
&= \frac{1}{|B(0, 1)|^p}\int_{|\tilde{y}_1| \leq 1}\dots\int_{|\tilde{y}_p| \leq 1}  \E \left[ \prod_{i=1}^p  \left( \log \EE_{\eps^{-1}{\tilde{y}_ir_\eps}}\big[\Psi_{s}^{\beta_{\eps}, 0}\big] - \E  \log \EE_{\eps^{-1}{\tilde{y}_ir_\eps}}\big[\Psi_{s}^{\beta_{\eps}, 0}\big] \right)\right] \dd \tilde{y}_1 \dots \dd \tilde{y}_p.
\end{aligned}\label{eq:pthmom}
\end{equation}

Again the uniform integrability is guaranteed by Corollary~\ref{co:logmom}. Using Theorem~\ref{thm:kpzpw}, the continuous mapping theorem and the uniform integrability, for any $(\tilde{y}_{i})_{1 \leq i \leq p}$, 
\[
 \E \left[ \prod_{i=1}^p  \left( \log \EE_{\eps^{-1}{\tilde{y}_ir_\eps}}\big[\Psi_{s}^{\beta_{\eps}, 0}\big] - \E  \log \EE_{\eps^{-1}{\tilde{y}_ir_\eps}}\big[\Psi_{s}^{\beta_{\eps}, 0}\big] \right)\right] \to \E  \left[ \prod_{i=1}^p Y_i \right], \quad \text{ as } \eps \to 0,\] where $(Y_i)_{1 \leq i \leq p}$ are again the jointly Gaussian random variables defined in Theorem~\ref{thm:kpzpw} with $\zeta_{i,j} = \gamma$ for any $ i \neq j$.

Let $P_p^2$ be the set of all the pairings of $\{1, \dots, p\}$. By Wick's probability theorem, 
\[
\E  \left[ \prod_{i=1}^p Y_i \right] = \sum_{s \in P_p^2} \prod_{\{i, j\} \in s} \Cov [Y_i, Y_j].
\]
In particular, 
\begin{equation}
\E  \left[ \prod_{i=1}^p Y_i \right] = \begin{cases}
0, & \quad \text{if } p \text{ is odd,}\\
\sigma_\gamma^p(p-1)!!, & \quad \text{if } p \text{ is even.}
\end{cases}\label{eq:wick}\nonumber
\end{equation}
Here we use $n!!$ to denote the double factorial.

By Lebesgue's dominated convergence theorem, \eqref{eq:pthmom} converges to  $\E  \left[ \prod_{i=1}^p Y_i \right]$ as $\eps \to 0$. Thus for any $p \geq 1$, the $p$-th moment of  \[\frac{1}{|B(0, r_{\eps})|}\int_{|y| \leq r_{\eps}} \log \EE_{\eps^{-1}y}\big[\Psi_{s}^{\beta_{\eps}, 0}\big] \dd y\] converges to the $p$-th moment of the Gaussian distribution $ \mathcal{N}(0, \sigma_{\gamma}^2) - \frac{1}{2}\log\left(\frac{2\pi}{2\pi - \beta^2}\right)$ as $\eps \to 0$. This implies \eqref{eq:localave}.

\section{Proof of Corollary~\ref{co:localfield}}
\added{By the equality in law \eqref{eq:kpzfk},} we \added{can} prove Corollary~\ref{co:localfield} by showing that as $\eps \to 0$,
\begin{equation}\label{eq:corl2}
\E \left[\int_{\R^2} \left(\frac{1}{|B(x, r_\eps)|} \int_{|y-x|\leq{r_\eps}} \log \EE_{\eps^{-1}y}\Big[\Psi_{\eps^{-2}t}^{\beta_{\eps}, h_0}\Big] \dd y -\left[\bar{h}(t, x)- \frac{1}{2}\log\left(\frac{2\pi}{2\pi - \beta^2}\right)\right] \right)g(x)\dd x\right]^2 \to 0.
\end{equation}

In order to prove \eqref{eq:corl2}, we prove that as $\eps \to 0$,
\begin{equation}\label{eq:corproof1}
\begin{aligned}
&\left\|\int_{\R^2} \left(\frac{1}{|B(x, r_\eps)|} \int_{|y-x|\leq{r_\eps}} \log \EE_{\eps^{-1}y}\Big[\Psi_{\eps^{-2}t}^{\beta_{\eps}, h_0}\Big] \dd y-\E\left[ \frac{1}{|B(x, r_\eps)|} \int_{|y-x|\leq{r_\eps}} \log \EE_{\eps^{-1}y}\Big[\Psi_{\eps^{-2}t}^{\beta_{\eps}, h_0}\Big] \dd y\right] \right)g(x)\dd x\right\|_{L^2(\Omega)} \\
&\quad \to 0,
\end{aligned}
\end{equation}
and 
\begin{equation}\label{eq:corproof2}
\begin{aligned}
\int_{\R^2} \left(\E\left[ \frac{1}{|B(x, r_\eps)|} \int_{|y-x|\leq{r_\eps}} \log \EE_{\eps^{-1}y}\Big[\Psi_{\eps^{-2}t}^{\beta_{\eps}, h_0}\Big] \dd y\right] -\left[\bar{h}(t, x)- \frac{1}{2}\log\left(\frac{2\pi}{2\pi - \beta^2}\right)\right] \right)g(x)\dd x \to 0.
\end{aligned}
\end{equation}

Let $\tilde{\mathfrak{h}}^{\eps, \gamma}(t,x) :=  \frac{1}{|B(x, r_\eps)|} \int_{|y-x|\leq{r_\eps}} \log \EE_{\eps^{-1}y}\Big[\Psi_{\eps^{-2}t}^{\beta_{\eps}, h_0}\Big] \dd y$. By using again the decomposition \added{\eqref{eq:decomp}, we can prove}\deleted{  as in the proof of Theorem~\ref{thm:main}, we can show} that for any $x_1 \neq x_2$, $x_1, x_2 \in \R^2$, as $\eps \to 0$, \begin{equation*}
\begin{aligned}
&\Cov\left[\tilde{\mathfrak{h}}^{\eps, \gamma}(t,x_1),\tilde{\mathfrak{h}}^{\eps, \gamma}(t,x_2)\right]\\ &= \E  \left(\prod_{i=1,2}\left[ \frac{1}{|B(x_i, r_\eps)|} \int_{|y_i-x_i|\leq{r_\eps}}  \left(\log \EE_{\eps^{-1}y_i}\Big[\Psi_{\eps^{-2}t}^{\beta_{\eps}, h_0}\Big]- \E \log \EE_{\eps^{-1}y_i}\Big[\Psi_{\eps^{-2}t}^{\beta_{\eps}, h_0}\Big] \right) \dd y_i\right]\right)\\
& = \E  \left(\prod_{i=1,2}\left[ \frac{1}{|B(x_i, 1)|} \int_{|\tilde{y}_i|\leq{1}}  \left(\log \EE_{\eps^{-1}({\tilde{y}_ir_\eps + x_i})}\Big[\Psi_{\eps^{-2}t}^{\beta_{\eps}, h_0}\Big]- \E \log \EE_{\eps^{-1}({\tilde{y}_i r_\eps + x_i})}\Big[\Psi_{\eps^{-2}t}^{\beta_{\eps}, h_0}\Big] \right) \dd \tilde{y}_i\right]\right)\\
&\to 0.
\end{aligned}
\end{equation*}
\added{More precisely, we can apply \eqref{eq:decomp} to split $\tilde{\mathfrak{h}}^{\eps, \gamma}(t,x_i)$ into an ``almost deterministic'' part on $[0,t-o(1))$ and a ``random'' part on $[t-o(1),t]$ again. We can then use the Lebesgue's dominated convergence theorem to prove that the covariance of the ``random'' part converges to zero as $\eps \to 0$, where we appeal to Corollary~\ref{co:logmom} with $p=1$ for the uniform integrability.}\deleted{where Corollary~\ref{co:logmom} with $p=1$ gives the validity of using Lebesgue's dominated convergence theorem to prove the convergence.}
If $h_0=0$, the above convergence is a direct consequence of Theorem~\ref{thm:kpzpw} and Corollary~\ref{co:logmom}.

We can now prove \eqref{eq:corproof1}. In fact,
\begin{equation}\label{eq:l2avereagedfield}
\begin{aligned}
&\E\left[\int_{\R^2} \left(\frac{1}{|B(x, r_\eps)|} \int_{|y-x|\leq{r_\eps}} \log \EE_{\eps^{-1}y}\Big[\Psi_{\eps^{-2}t}^{\beta_{\eps}, h_0}\Big] - \E \log \EE_{\eps^{-1}y}\Big[\Psi_{\eps^{-2}t}^{\beta_{\eps}, h_0}\Big] \dd y \right)g(x)\dd x\right]^2\\
&= \int_{\R^2} \int_{\R^2} \Cov\left[\tilde{\mathfrak{h}}^{\eps, \gamma}(t,x_1),\tilde{\mathfrak{h}}^{\eps, \gamma}(t,x_2)\right]g(x_1)g(x_2)\dd x_1 \dd x_2.
\end{aligned}
\end{equation}
Again by using Corollary~\ref{co:logmom}, we can apply Lebesgue's dominated convergence theorem to show that \eqref{eq:l2avereagedfield} converges to zero as $\eps \to 0$.

To prove \eqref{eq:corproof2}, we note that Corollary~\ref{co:logmom} also guarantees the uniform integrability of $\frac{1}{|B(x, r_\eps)|} \int_{|y-x|\leq{r_\eps}} \log \EE_{\eps^{-1}y}\Big[\Psi_{\eps^{-2}t}^{\beta_{\eps}, h_0}\Big] \dd y$ for any fixed $t>0, x\in \R^2$. Thus for any fixed $x \in \R^2$, as $\eps \to 0$,
\[ \E \left[\frac{1}{|B(x, r_\eps)|} \int_{|y-x|\leq{r_\eps}} \log \EE_{\eps^{-1}y}\Big[\Psi_{\eps^{-2}t}^{\beta_{\eps}, h_0}\Big] \dd y \right] \to \bar{h}(t, x)- \frac{1}{2}\log\left(\frac{2\pi}{2\pi - \beta^2}\right).
\]
Now by using again Corollary~\ref{co:logmom} with $p=1$, we can apply Lebesgue's dominated convergence theorem to obtain \eqref{eq:corproof2}.

\appendix
\section{Proof of Lemma~\ref{le:condmean}}\label{ap:meanlemma}

From \eqref{eq:averagedpol}, we have
\begin{equation}\label{eq:cdmeandiff}
\begin{aligned}
&\mathbf{{E}}^{\mathscr{F}_{s}} \left\{  \frac{ \EE_{\eps^{-1}x}\big[\Psi_{\eps^{-2}t}^{\beta_{\eps}, h_0}\big]}{\EE_{\eps^{-1}x}\big[\Psi_{\eps^{-2(1-a_\eps)}t}^{\beta_{\eps}, 0}\big]} \right\} - \exp\big[\bar{h}(t- \eps^{2a_{\eps}}t,x)\big] \\
&= \mathbf{{E}}^{\mathscr{F}_{s}} \left\{ \frac{\int_{z \in \R^2} 
\rho_{s}(\eps^{-1}x,z) \EE_{0, \eps^{-1}x}^{s, z}\Big[\Psi_{s}^{\beta_{\eps}, 0} (B^B, \xi) \Big]\left\{\EE_z\Big[\Psi_{\eps^{-2}t-s}^{\beta_{\eps}, h_0} (\tilde{B}, \xi') \Big] -  \exp[{\bar{h}(t- \eps^{2a_{\eps}}t,x)}] \right\}\dd z}{\int_{z \in \R^2} \rho_s(\eps^{-1}x,z)\EE_{0, \eps^{-1}x}^{s, z}\Big[\Psi_{s}^{\beta_{\eps}, 0} (B^B, \xi) \Big]\dd z}  \right\}\\
&= \int_{z \in \R^2}   \mathbf{{E}}^{\mathscr{F}_{s}} \left\{ \frac{\rho_{s}(\eps^{-1}x,z) \EE_{0, \eps^{-1}x}^{s, z}\Big[\Psi_{s}^{\beta_{\eps}, 0} (B^B, \xi) \Big]\left\{\EE_z\Big[\Psi_{\eps^{-2}t-s}^{\beta_{\eps}, h_0} (\tilde{B}, \xi') \Big] -  \exp[{\bar{h}(t- \eps^{2a_{\eps}}t,x)}]\right\}}{\int_{z \in \R^2} \rho_s(\eps^{-1}x,z)\EE_{0, \eps^{-1}x}^{s, z}\Big[\Psi_{s}^{\beta_{\eps}, 0} (B^B, \xi) \Big]\dd z} \right\}\dd z.
\end{aligned}
\end{equation}
By \added{the }definition \added{above},\added{ \[\xi'(r, \cdot) =\xi(s+r, \cdot) \]  is a time shift of $\xi$ by $s$.}
Thus $\{\xi'(r, \cdot)\}_{r \geq 0}$ is independent of the filtration $\mathscr{F}_{s}$. Since $B^B$ and $\tilde{B}$ are independent\replaced{, and}{.} $\EE_{0, \eps^{-1}x}^{s, z} \Big[\Psi_{s}^{\beta_{\eps}, 0} (B^B, \xi) \Big]$ is measurable w.r.t. $\mathscr{F}_{s}$\added{, the }above \eqref{eq:cdmeandiff} equals to
\begin{equation}
\begin{aligned}
 & \int_{z \in \R^2}  \rho_{s}(\eps^{-1}x,z) \mathbf{{E}}^{\mathscr{F}_{s}} \left\{ \frac{\EE_{0, \eps^{-1}x}^{s, z} \Big[\Psi_{s}^{\beta_{\eps}, 0} (B^B, \xi) \Big]}{\int_{z \in \R^2} \rho_s(\eps^{-1}x,z)\EE_{0, \eps^{-1}x}^{s, z}\Big[\Psi_{s}^{\beta_{\eps}, 0} (B^B, \xi) \Big]\dd z}
  \right\} \\
& \quad \E \left\{\EE_z\Big[\Psi_{\eps^{-2}t-s}^{\beta_{\eps}, h_0} (\tilde{B}, \xi') \Big] -  \exp[\bar{h}(t- \eps^{2a_{\eps}}t,x)]\right\} \dd z\\
& =  \int_{z \in \R^2}  \rho_{s}(\eps^{-1}x,z) \frac{\EE_{0, \eps^{-1}x}^{s, z} \Big[\Psi_{s}^{\beta_{\eps}, 0} (B^B, \xi) \Big]}{\EE_{\eps^{-1}x}\big[\Psi_{s}^{\beta_{\eps}, 0}\big] }  \E \left\{\EE_z\Big[\Psi_{\eps^{-2}t-s}^{\beta_{\eps}, h_0} (\tilde{B}, \xi') \Big] -  \exp[\bar{h}(t- \eps^{2a_{\eps}}t,x)] \right\} \dd z,
\end{aligned}\nonumber
\end{equation}
\added{where we applied \eqref{eq:denominatorterm} to the denominator}.
Since $ \E \left[
\Psi_{s}^{\beta_{\eps}, 0} (\tilde{B}, \xi')\right]=1$, through scaling invariance of the Brownian motion $\tilde{B}$, the above equals to
\begin{equation*}
\begin{aligned}
&\int_{z \in \R^2}  \rho_{s}(\eps^{-1}x,z) \frac{\EE_{0, \eps^{-1}x}^{s, z} \Big[\Psi_{s}^{\beta_{\eps}, 0} (B^B, \xi) \Big]}{\EE_{\eps^{-1}x}\big[\Psi_{s}^{\beta_{\eps}, 0}\big] } \\&\quad\E\left\{ \EE_{ z}\left[\exp \big[h_0(\eps\tilde{B}_{\eps^{-2}t-s})\big]
\Psi_{\eps^{-2}t-s }^{\beta_{\eps}, 0} (\tilde{B}, \xi')\right]-  \exp[{\bar{h}(t- \eps^{2a_{\eps}}t,x)}]\right\}\dd z\\
 =& \int_{z \in \R^2}  \rho_{s}(\eps^{-1}x,z) \frac{\EE_{0, \eps^{-1}x}^{s, z} \Big[\Psi_{s}^{\beta_{\eps}, 0} (B^B, \xi) \Big]}{\EE_{\eps^{-1}x}\big[\Psi_{s}^{\beta_{\eps}, 0}\big] } \left(\EE_{\eps z}\left[\exp[{h_0(\tilde{B}_{t-\eps^{2a_\eps}t})}]
\right]-  \exp[{\bar{h}(t - \eps^{2a_{\eps}}t, x)}]\right)\dd z\\
 = &\int_{z' \in \R^2} \rho_{s}(\eps^{-1}x,\eps^{-1}z') \frac{\EE_{0, \eps^{-1}x}^{s, \eps^{-1}z'} \Big[\Psi_{s}^{\beta_{\eps}, 0} (B^B, \xi) \Big]}{\EE_{\eps^{-1}x}\big[\Psi_{s}^{\beta_{\eps}, 0}\big] } \left(\EE_{z'}\left[\exp[{h_0(\tilde{B}_{t-\eps^{2a_\eps}t})}]
\right]-  \exp[{\bar{h}(t - \eps^{2a_{\eps}}t, x)}]\right) \eps^{-2}\dd z'.
\end{aligned}
\end{equation*}
We applied a change of variable $\eps z \to z'$ in the last equation.

Using \replaced{Minkowski’s integral inequality}{Minkowski inequality} and the above computation, we have
\begin{equation}
\begin{aligned}
&\left \|\mathbf{{E}}^{\mathscr{F}_{s}} \left\{  \frac{ \EE_{\eps^{-1}x}\big[\Psi_{\eps^{-2}t}^{\beta_{\eps}, h_0}\big]}{\EE_{\eps^{-1}x}\big[\Psi_{s}^{\beta_{\eps}, 0}\big]} \right\} - \exp\big[\bar{h}(t- \eps^{2a_{\eps}}t,x)\big]\right\|_{L^2(\Omega)}\\ & \leq
 \int_{z' \in \R^2}  \rho_{s}(\eps^{-1}x,\eps^{-1}z') 
 \left|\EE_{z'}\left[\exp[{h_0(\tilde{B}_{t-\eps^{2a_\eps}t})}]
\right]-  \exp[{\bar{h}(t - \eps^{2a_{\eps}}t, x)}]\right|\\&\quad
 \left\|  \frac{\EE_{0, \eps^{-1}x}^{s, \eps^{-1}z'} \Big[\Psi_{s}^{\beta_{\eps}, 0} (B^B, \xi) \Big]}{\EE_{\eps^{-1}x}\big[\Psi_{s}^{\beta_{\eps}, 0}\big] }\right\|_{L^2(\Omega)} \eps^{-2}\dd z'.
\end{aligned}\label{eq:mink}
\end{equation}

By  H\"older's inequality, for any $p, q > 2$ and $1/p + 1/q = 1/2$,
 \[ \left\|  \frac{\EE_{0, \eps^{-1}x}^{s, \eps^{-1}z'} \Big[\Psi_{s}^{\beta_{\eps}, 0} (B^B, \xi) \Big]}{\EE_{\eps^{-1}x}\big[\Psi_{s}^{\beta_{\eps}, 0}\big] }\right\|_{L^2(\Omega)} \leq \left\| \EE_{0, \eps^{-1}x}^{s, \eps^{-1}z'} \Big[\Psi_{s}^{\beta_{\eps}, 0} (B^B, \xi) \Big]\right\|_{L^p(\Omega)}  \left\|\frac{1}{\EE_{\eps^{-1}x}\big[\Psi_{s}^{\beta_{\eps}, 0}\big]}  \right\|_{L^q(\Omega)}.
 \]
 By Lemma~\ref{le:higher} and Lemma~\ref{le:negmom}, if $ p < p_{\beta}$, there exists some $\tilde{C}_{\beta}>0$ such that for any $z' \in \R^2$ and $0<\eps\leq1$, 
\[ \left\|  \frac{\EE_{0, \eps^{-1}x}^{s, \eps^{-1}z'} \Big[\Psi_{s}^{\beta_{\eps}, 0} (B^B, \xi) \Big]}{\EE_{\eps^{-1}x}\big[\Psi_{s}^{\beta_{\eps}, 0}\big] }\right\|_{L^2(\Omega)} \leq\tilde {C}_{\beta}.\]
Therefore, \eqref{eq:mink} is bounded by 
\begin{equation*}
\begin{aligned}
&\tilde{C}_{\beta}
\int_{z' \in \R^2}  \rho_{s}(\eps^{-1}x,\eps^{-1}z') 
 \left|\EE_{z'}\left[\exp[{h_0(\tilde{B}_{t-\eps^{2a_\eps}t})}]
\right]-  \exp[\bar{h}(t - \eps^{2a_{\eps}}t, x)]\right|
\eps^{-2}\dd z' \\
& =\tilde{C}_{\beta}
\int_{z' \in \R^2}  \rho_{\eps^2 s}(x, z') 
 \left|\EE_{z'}\left[\exp[{h_0(\tilde{B}_{t-\eps^{2a_\eps}t})}]
\right]-  \exp[{\bar{h}(t - \eps^{2a_{\eps}}t, x)}]\right|
\dd z'.
\end{aligned}
\end{equation*}
By Feynman-Kac formula, $\exp[\bar{h}(t - \eps^{2a_{\eps}}t, x)] = \EE_x\left[\exp[{h_0({B}_{t-\eps^{2a_\eps}t})}]\right]$ for some Brownian motion $B$ starting at $x$. Let $B^0$ be another Brownian motion starting at $0$. For any $a, b \in \R$, $|e^{a} - e^b| \leq (e^a + e^b)|a-b|$. Now since $h_0$ is bounded and Lipschitz continuous, there exists some $C>0$, such that
\begin{equation*}
\begin{aligned}
 &\left|\EE_{z'}\left[\exp\left[{h_0(\tilde{B}_{t-\eps^{2a_\eps}t})}\right]
\right]-  \exp\left[{h(t - \eps^{2a_{\eps}}t, x)}\right]\right|  \\&= \left|\EE_{z'}\left[\exp[{h_0(\tilde{B}_{t-\eps^{2a_\eps}t})}]
\right]-  \EE_x\left[\exp[{h_0({B}_{t-\eps^{2a_\eps}t})}]\right]\right|\\
& \leq  \EE_0 \left|\exp[{h_0(z' + {{B}}^0_{t-\eps^{2a_\eps}t})}] - \exp[{h_0(x +{{B}}^0_{t-\eps^{2a_\eps}t})}]\right| \\ &\leq C |z'-x|.\end{aligned}
\end{equation*}
Then since
\[
\int_{z' \in \R^2}  \rho_{\eps^2 s}(x, z') |z'-x| \dd z'  = 2\pi \int_{0}^{+\infty} \frac{r^2}{2\pi \eps^2 s} e^{- \frac{r^2}{2 \eps^2 s}} \dd r = \sqrt{\frac{\pi}{2}\eps^2s},\]
and $s = \eps^{-2(1-a_\eps)}t$, we have proved \eqref{eq:condmean}.

\section{Proof of Lemma~\ref{le:condvar}}\label{ap:varlemma}

Again, we shall use the facts that $\{\xi'(r, \cdot)\}_{r \geq 0}$ is independent of the filtration $\mathscr{F}_{s}$, $B^B$ and $\tilde{B}$ are independent, and $\EE_{0, \eps^{-1}x}^{s, z} \Big[\Psi_{s}^{\beta_{\eps}, 0}  \Big] :=\EE_{0, \eps^{-1}x}^{s, z} \Big[\Psi_{s}^{\beta_{\eps}, 0} (B^B, \xi) \Big]$ is measurable w.r.t. $\mathscr{F}_{s}$. 
With the same approach as in the proof of Lemma~\ref{le:condmean}, we have
\begin{equation}
\begin{aligned}
&\mathbf{{V}ar}^{\mathscr{F}_{s}} \left\{  \frac{ \EE_{\eps^{-1}x}\big[\Psi_{\eps^{-2}t}^{\beta_{\eps}, h_0}\big]}{\EE_{\eps^{-1}x}\big[\Psi_{s}^{\beta_{\eps}, 0}\big]} \right\}  =
\mathbf{{V}ar}^{\mathscr{F}_{s}} \left\{  \frac{\int_{z \in \R^2} 
\rho_{s}(\eps^{-1}x,z) \EE_{0, \eps^{-1}x}^{s, z}\Big[\Psi_{s}^{\beta_{\eps}, 0} \Big]\EE_z\Big[\Psi_{\eps^{-2}t-s}^{\beta_{\eps}, h_0} (\tilde{B}, \xi') \Big] \dd z}{\int_{z \in \R^2} \rho_s(\eps^{-1}x,z)\EE_{0, \eps^{-1}x}^{s, z}\Big[\Psi_{s}^{\beta_{\eps}, 0} \Big]\dd z}  \right\}\\
& = \frac{\int_{z \in \R^2} \int_{z' \in \R^2} 
\rho_{s}(\eps^{-1}x,z)
\rho_{s}(\eps^{-1}x,z') \EE_{0, \eps^{-1}x}^{s, z}\Big[\Psi_{s}^{\beta_{\eps}, 0}\Big]\EE_{0, \eps^{-1}x}^{s, z'}\Big[\Psi_{s}^{\beta_{\eps}, 0}\Big] \mathcal{C}_{z,z'}^{\eps} \dd z \dd z'}{\int_{z \in \R^2} \int_{z' \in \R^2} 
\rho_{s}(\eps^{-1}x,z)
\rho_{s}(\eps^{-1}x,z') \EE_{0, \eps^{-1}x}^{s, z}\Big[\Psi_{s}^{\beta_{\eps}, 0}\Big]\EE_{0, \eps^{-1}x}^{s, z'}\Big[\Psi_{s}^{\beta_{\eps}, 0}\Big]\dd z \dd z' },
%& = \frac{\int_{z \in \R^2} \int_{z' \in \R^2} 
%\rho_{s}(\eps^{-1}x,z)
%\rho_{s}(\eps^{-1}x,z') \EE_{0, \eps^{-1}x}^{s, z}\Big[\Psi_{s}^{\beta_{\eps}, 0}\Big]\EE_{0, \eps^{-1}x}^{s, z'}\Big[\Psi_{s}^{\beta_{\eps}, 0}\Big] \mathcal{C}_{z,z'}^{\eps} \dd z \dd z'}{\left(\EE_{\eps^{-1}x}\big[\Psi_{s}^{\beta_{\eps}, 0}\big]\right)^2},
\end{aligned}\label{eq:averagedvar}
\end{equation}
where \begin{equation*}
\mathcal{C}_{z,z'}^{\eps} := \mathbf{{C}ov}^{\mathscr{F}_{s}} \left( \EE_z\Big[\Psi_{\eps^{-2}t-s}^{\beta_{\eps}, h_0} (\tilde{B}, \xi') \Big], \EE_{z'}\Big[\Psi_{\eps^{-2}t-s}^{\beta_{\eps}, h_0} (\tilde{B}, \xi') \Big] \right).
\end{equation*}
Since $\{\xi'(r, \cdot)\}_{r \geq 0}$ is independent of the filtration $\mathscr{F}_{s}$, 
\[
\mathbf{{C}ov}^{\mathscr{F}_{s}} \left( \EE_z\Big[\Psi_{\eps^{-2}t-s}^{\beta_{\eps}, h_0} (\tilde{B}, \xi') \Big], \EE_{z'}\Big[\Psi_{\eps^{-2}t-s}^{\beta_{\eps}, h_0} (\tilde{B}, \xi') \Big] \right)= \Cov \left( \EE_z\Big[\Psi_{\eps^{-2}t-s}^{\beta_{\eps}, h_0} (\tilde{B}, \xi') \Big], \EE_{z'}\Big[\Psi_{\eps^{-2}t-s}^{\beta_{\eps}, h_0} (\tilde{B}, \xi') \Big] \right).
\]
For fixed $z, z' \in \R^2$ and $\eps>0$, $\mathcal{C}_{z,z'}^{\eps}$ is a deterministic value, and
\begin{equation*}
\begin{aligned}\mathcal{C}_{z,z'}^{\eps} &=\Cov \left( \EE_z\Big[\Psi_{\eps^{-2}t-s}^{\beta_{\eps}, h_0} (\tilde{B}, \xi') \Big], \EE_{z'}\Big[\Psi_{\eps^{-2}t-s}^{\beta_{\eps}, h_0} (\tilde{B}, \xi') \Big] \right) \\
&= \Cov  \Big(\EE_{ z}\left[\exp \big[h_0(\eps\tilde{B}^1_{\eps^{-2}t-s})\big]
\Psi_{\eps^{-2}t-s}^{\beta_{\eps}, 0} (\tilde{B}^1, \xi')\right],  \EE_{z'}\left[\exp \big[h_0(\eps\tilde{B}^2_{\eps^{-2}t-s})\big]
\Psi_{s}^{\beta_{\eps}, 0} (\tilde{B}^2, \xi')\right]\Big)\\
& = \E \Big\{ \EE_{ z}\left[\exp \big[h_0(\eps\tilde{B}^1_{\eps^{-2}t-s})\big]\left(
\Psi_{\eps^{-2}t-s}^{\beta_{\eps}, 0} (\tilde{B}^1, \xi') - 1 \right)\right]\\&\quad \EE_{ z'}\left[\exp \big[h_0(\eps\tilde{B}^2_{\eps^{-2}t-s})\big]\left(
\Psi_{\eps^{-2}t-s}^{\beta_{\eps}, 0} (\tilde{B}^2, \xi') - 1 \right)\right] \Big\}\\
& = \EE_z \otimes \EE_{z'} \Big\{ \exp \big[h_0(\eps\tilde{B}^1_{\eps^{-2}t-s}) +h_0(\eps\tilde{B}^2_{\eps^{-2}t-s}) \big]   \E \left(
\Psi_{\eps^{-2}t-s}^{\beta_{\eps}, 0} (\tilde{B}^1, \xi') - 1 \right) \left(
\Psi_{\eps^{-2}t-s}^{\beta_{\eps}, 0} (\tilde{B}^2, \xi') - 1 \right) \Big\}\\
&= \EE_z \otimes \EE_{z'} \Big\{ \exp \big[h_0(\eps\tilde{B}^1_{\eps^{-2}t-s}) +h_0(\eps\tilde{B}^2_{\eps^{-2}t-s}) \big]   \big(\exp \big[{\beta^2_{\eps}}\int_0^{\eps^{-2}t-s}V(\tilde{B}^1_r - \tilde{B}^2_r ) \dd r \big]-1\big)\Big\},
\end{aligned}
\end{equation*}
\added{where $V(x)= \phi * \phi(x)$ as defined in \eqref{eq:defofV}.}
Note that we \replaced{can}{could} bound \[\EE_z \otimes \EE_{z'}   \left(\exp \big[{\beta^2_{\eps}}\int_0^{\eps^{-2}t-s}V(\tilde{B}^1_r - \tilde{B}^2_r ) \dd r \big]-1\right)\] uniformly by some constant as in \eqref{eq:ucov}. Since $h_0$ is also bounded, there exists some $0 < C^{(1)}_{\beta}< \infty$, such that
\begin{equation}
\sup_{\eps \leq 1}\sup_{z, z' \in \R^2}\mathcal{C}_{z,z'}^{\eps} \leq C^{(1)}_{\beta}. \label{eq:unicovbd}
\end{equation}

From \eqref{eq:averagedvar},  $\mathbf{{V}ar}^{\mathscr{F}_{s}} \left\{  \frac{ \EE_{\eps^{-1}x}\big[\Psi_{\eps^{-2}t}^{\beta_{\eps}, h_0}\big]}{\EE_{\eps^{-1}x}\big[\Psi_{s}^{\beta_{\eps}, 0}\big]} \right\}$
is a weighted average of $C_{z,z'}^{\eps}$ over $z, z' \in \R^2$. Therefore, by \eqref{eq:unicovbd}, for $\eps \leq 1$, 
\begin{equation}
\mathbf{{V}ar}^{\mathscr{F}_{s}} \left\{  \frac{ \EE_{\eps^{-1}x}\big[\Psi_{\eps^{-2}t}^{\beta_{\eps}, h_0}\big]}{\EE_{\eps^{-1}x}\big[\Psi_{s}^{\beta_{\eps}, 0}\big]} \right\} \leq C^{(1)}_{\beta}, \quad \text{ almost surely.} \label{eq:bdconvar}
\end{equation}

Now we first assume that, for a particular realization of the noise field, 
$\EE_{\eps^{-1}x}\big[\Psi_{s}^{\beta_{\eps}, 0}\big] \leq b_{\eps}$
for some $b_{\eps} > 0$. We will choose an appropriate $b_{\eps}$ later so that $b_{\eps} = o(1)$.

By the Markov's inequality and  Lemma~\ref{le:negmom}, we have
\begin{equation}
\begin{aligned}
\Prob \left[\EE_{\eps^{-1}x}\big[\Psi_{s}^{\beta_{\eps}, 0}\big] \leq b_{\eps} \right] = \Prob \left[\left(\EE_{\eps^{-1}x}\big[\Psi_{s}^{\beta_{\eps}, 0}\big]\right)^{-1} \geq b_{\eps}^{-1}\right] \leq \E \left[\left(\EE_{\eps^{-1}x}\big[\Psi_{s}^{\beta_{\eps}, 0}\big]\right)^{-1} \right] b_{\eps} \leq C^{(2)}_{\beta}  b_{\eps},
\end{aligned}
\label{eq:lefttailpro}
\end{equation}
for some $ C^{(2)}_{\beta} >0$. By \eqref{eq:bdconvar} and \eqref{eq:lefttailpro} together, we have
\begin{equation}
\E \left[\mathbf{{V}ar}^{\mathscr{F}_{s}} \left\{  \frac{ \EE_{\eps^{-1}x}\big[\Psi_{\eps^{-2}t}^{\beta_{\eps}, h_0}\big]}{\EE_{\eps^{-1}x}\big[\Psi_{s}^{\beta_{\eps}, 0}\big]} \right\}; \EE_{\eps^{-1}x}\big[\Psi_{s}^{\beta_{\eps}, 0}\big] \leq b_{\eps} \right] \leq C^{(1)}_{\beta} C^{(2)}_{\beta}   b_{\eps}.
\label{eq:condmean1}
\end{equation}

We then assume that for a particular realization of the noise field, 
\[\EE_{\eps^{-1}x}\big[\Psi_{s}^{\beta_{\eps}, 0}\big] > b_{\eps}.\]
In this case, by \eqref{eq:averagedvar},
\begin{equation}
\begin{aligned}
&\mathbf{{V}ar}^{\mathscr{F}_{s}} \left\{  \frac{ \EE_{\eps^{-1}x}\big[\Psi_{\eps^{-2}t}^{\beta_{\eps}, h_0}\big]}{\EE_{\eps^{-1}x}\big[\Psi_{s}^{\beta_{\eps}, 0}\big]} \right\}\\
&\leq  b_{\eps}^{-2} \int_{z \in \R^2} \int_{z' \in \R^2} 
\rho_{s}(\eps^{-1}x,z)
\rho_{s}(\eps^{-1}x,z') \EE_{0, \eps^{-1}x}^{s, z}\Big[\Psi_{s}^{\beta_{\eps}, 0}\Big]\EE_{0, \eps^{-1}x}^{s, z'}\Big[\Psi_{s}^{\beta_{\eps}, 0}\Big] \mathcal{C}_{z,z'}^{\eps} \dd z \dd z'.
\label{eq:varcon2leq}
\end{aligned}
\end{equation}
We bound the expectation of the right-hand-side of \eqref{eq:varcon2leq}.
By H\"older's inequality and Lemma~\ref{le:2ndmombb}, there exists some $0 < C^{(3)}_\beta < \infty$, such that
\begin{equation}
\begin{aligned}
& b_{\eps}^{-2} \E  \int_{z \in \R^2} \int_{z' \in \R^2} 
\rho_{s}(\eps^{-1}x,z)
\rho_{s}(\eps^{-1}x,z') \EE_{0, \eps^{-1}x}^{s, z}\Big[\Psi_{s}^{\beta_{\eps}, 0}\Big]\EE_{0, \eps^{-1}x}^{s, z'}\Big[\Psi_{s}^{\beta_{\eps}, 0}\Big] \mathcal{C}_{z,z'}^{\eps} \dd z \dd z' \\
 &= b_{\eps}^{-2} \int_{z \in \R^2} \int_{z' \in \R^2} 
\rho_{s}(\eps^{-1}x,z)
\rho_{s}(\eps^{-1}x,z') \E \left\{\EE_{0, \eps^{-1}x}^{s, z}\Big[\Psi_{s}^{\beta_{\eps}, 0}\Big]\EE_{0, \eps^{-1}x}^{s, z'}\Big[\Psi_{s}^{\beta_{\eps}, 0}\Big] \right\}  \mathcal{C}_{z,z'}^{\eps} \dd z \dd z' \\
&\leq b_{\eps}^{-2}  C_\beta^{(3)} \int_{z \in \R^2} \int_{z' \in \R^2} 
\rho_{s}(\eps^{-1}x,z)
\rho_{s}(\eps^{-1}x,z') \mathcal{C}_{z,z'}^{\eps} \dd z \dd z' \\
&= b_{\eps}^{-2} C_\beta^{(3)} \int_{y \in \R^2} \int_{y' \in \R^2} 
\rho_{\eps^2s}(x, y)
\rho_{\eps^2s}(x, y') \mathcal{C}_{\eps^{-1}y,\eps^{-1}y'}^{\eps} \dd y \dd y',
\end{aligned}\label{eq:expcondvar2}
\end{equation}
with a change of variable $\eps z \to y$ and $\eps z' \to y'$ in the last equation.

Since we have the term $b_{\eps}^{-2} $ in \eqref{eq:expcondvar2}, the uniform bound of $ \mathcal{C}_{\eps^{-1}y,\eps^{-1}y'}^{\eps}$ solely will not give us a desired bound. We shall proceed as the following.

By the boundedness of $h_0$, there exists some constant $C_{h_0} >0$ such that 
\begin{equation}\label{eq:condvar2}
\begin{aligned}
&\quad \int_{y \in \R^2} \int_{y' \in \R^2} 
\rho_{\eps^2s}(x, y)
\rho_{\eps^2s}(x, y') \mathcal{C}_{\eps^{-1}y,\eps^{-1}y'}^{\eps} \dd y \dd y'\\
 &= \int_{y \in \R^2} \int_{y' \in \R^2} 
\rho_{\eps^2s}(x, y)
\rho_{\eps^2s}(x, y')  \\ 
& \EE_{\eps^{-1}y} \otimes \EE_{\eps^{-1}y'} \Big\{ \exp \big[h_0(\eps\tilde{B}^1_{\eps^{-2}t-s}) +h_0(\eps\tilde{B}^2_{\eps^{-2}t-s}) \big] \big(\exp \big[{\beta^2_{\eps}}\int_0^{\eps^{-2}t-s}V(\tilde{B}^1_r - \tilde{B}^2_r ) \dd r \big]-1\big)\Big\} \dd y \dd y'\\
&\leq C_{h_0}  \int_{y \in \R^2} \int_{y' \in \R^2} 
\rho_{\eps^2s}(x, y)
\rho_{\eps^2s}(x, y')
 \EE_{\eps^{-1}y} \otimes \EE_{\eps^{-1}y'}  \big(\exp \big[{\beta^2_{\eps}}\int_0^{\eps^{-2}t-s}V(\tilde{B}^1_r - \tilde{B}^2_r ) \dd r \big]-1\big) \dd y \dd y '\\
&= C_{h_0}  \int_{y \in \R^2} \int_{y' \in \R^2} 
\rho_{\eps^2s}(x, y)
\rho_{\eps^2s}(x, y')  \EE_{\eps^{-1}\frac{y-y'}{\sqrt{2}}}  \big(\exp \big[{\beta^2_{\eps}}\int_0^{\eps^{-2}t-s}V(\sqrt{2} B_r) \dd r \big]-1\big) \dd y \dd y '\\
&= C_{h_0}  \int_{\R^2} \int_{\R^2} 
\rho_{\eps^2s}(y)
\rho_{\eps^2s}(y')
 \EE_{\eps^{-1}\frac{y-y'}{\sqrt{2}}}  \big(\exp \big[{\beta^2_{\eps}}\int_0^{\eps^{-2}t-s}V(\sqrt{2} B_r) \dd r \big]-1\big) \dd y \dd y ' \\
&= C_{h_0}  \int_{\R^2}
\rho_{2\eps^2s}(y)
 \EE_{\eps^{-1}\frac{y}{\sqrt{2}}}  \big(\exp \big[{\beta^2_{\eps}}\int_0^{\eps^{-2}t-s}V(\sqrt{2} B_r) \dd r \big]-1\big) \dd y.
\end{aligned}
\end{equation}
By Taylor expansion, if we set $x_0 = \eps^{-1}\frac{y}{\sqrt{2}}$ and $s_0 = 0$, the last line equals to
\begin{equation}\label{eq:taylorexpansion}
\begin{aligned}
&C_{h_0}  \int_{\R^2}
\rho_{2\eps^2s}(y)\\
&\cdot\Big[\sum_{n=1}^{\infty}\beta_{\eps}^{2n} \int_{0 < s_1 < \dots < s_n< \eps^{-2}t-s}\int_{\R^{2n}} \prod_{i=1}^n V(\sqrt{2} x_i)\rho_{s_i - s_{i-1}}(x_{i-1}, x_i) \dd s_1 \dots \dd s_n \dd x_1 \dots \dd x_n\Big] \dd y.
\end{aligned}
\end{equation}

With $\int_{\R^2}{\eps^{-2}}V(\eps^{-1}x)\dd x =1 $,  we bound the $n=1$ term by the following:
\begin{equation*}
\begin{aligned}
&\int_{\R^2}
\rho_{2\eps^2s}(y)
\Big[\beta_{\eps}^2 \int_0^{\eps^{-2}t-s}\int_{\R^{2}} V(\sqrt{2} x_1)\rho_{s_1}(x_1,  \eps^{-1}\frac{y}{\sqrt{2}}) \dd s_1 \dd  x_1\Big] \dd y \\
&= \beta_{\eps}^2  \int_0^{\eps^{-2}t-s} \int_{\R^{2}}\int_{\R^{2}}\rho_{2\eps^2s}(y)2 \eps^2 \rho_{2 \eps^2 s_1}(\sqrt{2}\eps x_1,  {y}) V(\sqrt{2} x_1)\dd  x_1\dd y \dd s_1 \\
& = \beta_{\eps}^2  \int_0^{\eps^{-2}t-s} \int_{\R^{2}}2 \eps^2 \rho_{2 \eps^2 (s+s_1)}(\sqrt{2}\eps x_1) V(\sqrt{2} x_1)\dd  x_1 \dd s_1 \\
& = \beta_{\eps}^2  \int_0^{\eps^{-2}t-s} \int_{\R^{2}} \rho_{2 \eps^2 (s+s_1)}(x) V(\eps^{-1}x)\dd  x \dd s_1 = \beta_{\eps}^2  \int_{\eps^2s}^{t} \int_{\R^{2}} \rho_{2r}(x) \frac{1}{\eps^2} V(\eps^{-1}x)\dd  x \dd r\\
& \leq \beta_{\eps}^2 \int_{\eps^2s}^{t} \frac{1}{4 \pi r} \dd r 
= \frac{ \beta_{\eps}^2}{4 \pi} \log \frac{t}{\eps^2 s} = \frac{ \beta^2   \log \eps^{-2 a_{\eps}}}{4 \pi \log \eps^{-1}} =  \frac{ \beta^2  }{2\pi} a_{\eps}.
\end{aligned}
\end{equation*}

For higher order terms, we have
\begin{equation*}
\begin{aligned}
& \int_{\R^2}
\rho_{2\eps^2s}(y)
\Big[\beta_{\eps}^{2n} 
\int_{0 < s_1 < \dots < s_n< \eps^{-2}t-s}\int_{\R^{2n}}\prod_{i=1}^n V(\sqrt{2} x_i) \rho_{s_i - s_{i-1}}(x_{i-1}, x_i) \dd s_1 \dots \dd s_n \dd x_1 \dots \dd x_n\Big] \dd y\\
& =  \int_{\R^2}
\rho_{2\eps^2s}(y)
\Big[\beta_{\eps}^{2n} 
\int_{0 < s_1 < \dots < s_n< \eps^{-2}t-s}\int_{\R^{2n}}\prod_{i=1}^n V(\sqrt{2} x_i) \\
& \quad (2\eps^2)^n \rho_{2\eps^2(s_i - s_{i-1})}(\sqrt{2}\eps x_{i-1}, \sqrt{2}\eps x_i) \dd s_1 \dots \dd s_n \dd x_1 \dots \dd x_n\Big] \dd y\\
& = \int_{\R^2}
\rho_{2\eps^2s}(y)
\Big[\beta_{\eps}^{2n} 
\int_{0 < s_1 < \dots < s_n< \eps^{-2}t-s}\int_{\R^{2n}}\prod_{i=1}^n V(\frac{\tilde{x}_i}{\eps}) \rho_{2\eps^2(s_i - s_{i-1})}(\tilde{x}_{i-1},\tilde{x}_i) \dd s_1 \dots \dd s_n \dd \tilde{x}_1 \dots \dd \tilde{x}_n\Big] \dd y\\
& = \int_{\R^2}
\rho_{2\eps^2s}(y)
\Big[\beta_{\eps}^{2n} 
\int_{0 < \tilde{s}_1 < \dots < \tilde{s}_n< t-\eps^2s}\int_{\R^{2n}}\prod_{i=1}^n \frac{1}{\eps^2}V(\frac{\tilde{x}_i}{\eps}) \rho_{2(\tilde{s}_i - \tilde{s}_{i-1})}(\tilde{x}_{i-1}, \tilde{x}_i) \dd \tilde{s}_1 \dots \dd \tilde{s}_n \dd  \tilde{x}_1 \dots \dd \tilde{x}_n\Big] \dd y,
\end{aligned}
\end{equation*}
with $\tilde{s}_0 = 0$ and $\tilde{x}_0 = y$ by a change of variable.

Using again $\int_{\R^2}\frac{1}{\eps^2}V(\eps^{-1}x)\dd x =1 $, for any $\eps, u_1, u_2 >0$ and $z_1, z_2 \in \R^2$, we have
\begin{equation}
\begin{aligned}
&\int_{\R^2}\rho_{u_1}(z_1) \rho_{u_2}(z_2, z_1)\frac{1}{\eps^2}V(\eps^{-1}z_1)\dd z_1 \\
&=\int_{\R^2} \frac{1}{4\pi^2u_1u_2} \exp\left(-\frac{|z_1|^2}{2u_1}-\frac{|z_1-z_2|^2}{2u_2}\right) \frac{1}{\eps^2}V(\eps^{-1}z_1)\dd z_1\\ &= \rho_{u_1+u_2}(z_2)\int_{\R^2} \frac{2\pi (u_1 +u_2)}{4\pi^2u_1u_2} \exp\left(\frac{|z_2|^2}{2(u_1+u_2)}-\frac{|z_1|^2}{2u_1}-\frac{|z_1-z_2|^2}{2u_2}\right) \frac{1}{\eps^2}V(\eps^{-1}z_1)\dd z_1\\
& = \rho_{u_1+u_2}(z_2)\int_{\R^2} \frac{2\pi (u_1 +u_2)}{4\pi^2u_1u_2} \exp\left(- \frac{u_1+u_2}{2u_1u_2}(z_1 - \frac{u_1}{u_1+u_2}z_2)^2\right)  \frac{1}{\eps^2}V(\eps^{-1}z_1)\dd z_1\\
& \leq  \rho_{u_1+u_2}(z_2)\int_{\R^2} \frac{(u_1 +u_2)}{2\pi u_1u_2}  \frac{1}{\eps^2}V(\eps^{-1}z_1)\dd z_1 =  \frac{u_1 +u_2}{2\pi u_1u_2}  \rho_{u_1+u_2}(z_2).
\end{aligned}\label{eq:ineq1}
\end{equation}
Since $V$ is bounded by $\|V\|_{\infty}$, \added{we also have that}
\begin{equation}
\int_{\R^2}\rho_{u_1}(z_1) \rho_{u_2}(z_2, z_1)\frac{1}{\eps^2}V(\eps^{-1}z_1)\dd z_1 \leq \frac{\|V\|_{\infty}}{\eps^2} \rho_{u_1+u_2}(z_2).\label{eq:ineq2}
\end{equation}
Now if we integrate $y$ first, and then $x_1, \dots, x_n$ successively utilizing the bounds \eqref{eq:ineq1} and \eqref{eq:ineq2}, we will have  that
\begin{equation}
\begin{aligned}
&\int_{\R^2}
\rho_{2\eps^2s}(y)
\Big[\beta_{\eps}^{2n} 
\int_{0 < {s}_1 < \dots < {s}_n< t-\eps^2s}\int_{\R^{2n}}\prod_{i=1}^n \frac{1}{\eps^2}V(\eps^{-1}x_i) \rho_{2({s}_i - {s}_{i-1})}({x}_{i-1}, {x}_i) \dd {s}_1 \dots \dd {s}_n \dd  {x}_1 \dots \dd {x}_n\Big] \dd y \\ 
& \leq \left(\frac{\beta_{\eps}^2}{2\pi}\right)^{n}  \int_{0 < {s}_1 < \dots < {s}_n< t-\eps^2s}  \frac{1}{2s_n + 2\eps^2 s}
\prod_{i=1}^{n-1}\left[\left(\frac{1}{2(s_{i+1}-s_i)} + \frac{1}{2s_i + 2\eps^2 s}\right)\wedge\frac{2\pi\|V\|_{\infty}}{\eps^2}\right]  \dd {s}_1 \dots \dd {s}_n.
\end{aligned}\label{eq:prodl2}
\end{equation}

To bound the last integral, we integrate $s_1, \dots, s_n$ successively. For any $1 \leq i \leq n-1$, we have
\begin{equation*}
\begin{aligned}
& \int_{0}^{s_{i+1}} \left(\frac{1}{2s_{i+1}-2s_i} + \frac{1}{2s_i + 2\eps^2 s}\right)\wedge\frac{2\pi\|V\|_{\infty}}{\eps^2} \dd s_i \\ &\leq \int_{0}^{s_{i+1}} \frac{1}{2s_{i+1}-2s_i}  \wedge\frac{2\pi\|V\|_{\infty}}{\eps^2} 
\dd s_i + \frac{1}{2}\int_0^{t- \eps^2s}  \frac{1}{s_i + \eps^2 s} \dd s_i\\
& \leq  \int_{0}^{s_{i+1}} \frac{1}{2s_i}  \wedge\frac{2\pi\|V\|_{\infty}}{\eps^2} 
\dd s_i + \frac{1}{2}\int_0^{t- \eps^2s}  \frac{1}{s_i + \eps^2 s} \dd s_i  \\&\leq  \int_{0}^{t} \frac{1}{2s_i}  \wedge\frac{2\pi\|V\|_{\infty}}{\eps^2} 
\dd s_i + \frac{1}{2}\int_0^{t- \eps^2s}  \frac{1}{s_i + \eps^2 s} \dd s_i\\
& =\frac{1}{2} \log \frac{4\pi\|V\|_{\infty}t}{\eps^2} + \frac{1}{2} + \frac{1}{2}\log \frac{t}{\eps^2 s}\\& = \frac{1}{2} \log (4\pi\|V\|_{\infty}t) + \frac{1}{2} + \log \eps^{-1} + a_{\eps} \log \eps^{-1}. 
\end{aligned}
\end{equation*}
Thus \eqref{eq:prodl2} is bounded by
\begin{equation}\label{eq:lastboundchaos}
\begin{aligned}
&\quad \left(\frac{\beta_{\eps}^2}{2\pi}\right)^{n}\int_0^{t-\eps^2s} \frac{1}{2s_n + 2\eps^2 s}\left( \frac{1}{2} \log (4\pi\|V\|_{\infty}t) + \frac{1}{2} + \log \eps^{-1} + a_{\eps} \log \eps^{-1}\right)^{n-1} \dd s_n \\
& = \left(\frac{\beta_{\eps}^2}{2\pi}\right)^{n}\left( \frac{1}{2} \log (4\pi\|V\|_{\infty}t) + \frac{1}{2} + \log \eps^{-1} + a_{\eps} \log \eps^{-1}\right)^{n-1} a_{\eps} \log \eps^{-1}\\
& =  \left(\frac{\beta^2}{2\pi}\right)^{n}\left(1+a_{\eps}+\frac{C_0}{\log \eps^{-1}}\right)^{n-1}a_\eps,
\end{aligned}
\end{equation}
with some constant $C_0 \in \R$ depends only on $V$ and $t$.

We can now bound \eqref{eq:expcondvar2} by  {\eqref{eq:condvar2}--\eqref{eq:taylorexpansion} and \eqref{eq:prodl2}--\eqref{eq:lastboundchaos}}\deleted{the above results}. We have that
\begin{equation*}
\begin{aligned}
& b_{\eps}^{-2} C_\beta^{(3)} \int_{y \in \R^2} \int_{y' \in \R^2} 
\rho_{\eps^2s}(x, y)
\rho_{\eps^2s}(x, y') \mathcal{C}_{\eps^{-1}y,\eps^{-1}y'}^{\eps} \dd y \dd y' \\ &\leq b_{\eps}^{-2} C_\beta^{(3)} C_{h_0}  a_\eps \sum_{n=1}^{\infty} \left(\frac{\beta^2}{2\pi}\right)^{n}\left(1+a_{\eps}+\frac{C_0}{\log \eps^{-1}}\right)^{n-1}.
\end{aligned}
\end{equation*}
Since $a_\eps \to 0$ and $\frac{C_0}{\log \eps^{-1}} \to 0$ as $\eps \to 0$, the above sum is bounded when $\eps$ is sufficiently small and $\beta < \sqrt{2\pi}$. In particular, there exist constants $C_\beta^{(4)}>0$ and $\eps_0>0$, such that when $0< \eps < \eps_0$, the infinite sum equals to \[A_{\eps} :=\sum_{n=1}^{\infty} \left(\frac{\beta^2}{2\pi}\right)^{n}\left(1+a_{\eps}+\frac{C_0}{\log \eps^{-1}}\right)^{n-1} =  \frac{\frac{\beta^2}{2\pi}}{1- (\frac{\beta^2}{2\pi})(1+a_{\eps}+\frac{C_0}{\log \eps^{-1}})} \leq C_\beta^{(4)}.\]

Now by \eqref{eq:varcon2leq} and \eqref{eq:expcondvar2}, we have 
\begin{equation}
\begin{aligned}
&\E \left[\mathbf{{V}ar}^{\mathscr{F}_{s}} \left\{  \frac{ \EE_{\eps^{-1}x}\big[\Psi_{\eps^{-2}t}^{\beta_{\eps}, h_0}\big]}{\EE_{\eps^{-1}x}\big[\Psi_{s}^{\beta_{\eps}, 0}\big]} \right\}; \EE_{\eps^{-1}x}\big[\Psi_{s}^{\beta_{\eps}, 0}\big] > b_{\eps} \right] \leq  C_{h_0} C_{\beta}^{(3)} C_\beta^{(4)}a_\eps b_{\eps}^{-2}.
\label{eq:condmean2}\nonumber
\end{aligned}
\end{equation}

Together with $\eqref{eq:condmean1}$, we conclude that when $\eps < \eps_0$, 
\begin{equation*}
\begin{aligned}
& \E \left[\mathbf{{V}ar}^{\mathscr{F}_{s}} \left\{  \frac{ \EE_{\eps^{-1}x}\big[\Psi_{\eps^{-2}t}^{\beta_{\eps}, h_0}\big]}{\EE_{\eps^{-1}x}\big[\Psi_{s}^{\beta_{\eps}, 0}\big]} \right\}\right] \\ &=  \E \left[\mathbf{{V}ar}^{\mathscr{F}_{s}} \left\{  \frac{ \EE_{\eps^{-1}x}\big[\Psi_{\eps^{-2}t}^{\beta_{\eps}, h_0}\big]}{\EE_{\eps^{-1}x}\big[\Psi_{s}^{\beta_{\eps}, 0}\big]} \right\}; \EE_{\eps^{-1}x}\big[\Psi_{s}^{\beta_{\eps}, 0}\big] \leq b_{\eps} \right] + \E \left[\mathbf{{V}ar}^{\mathscr{F}_{s}} \left\{  \frac{ \EE_{\eps^{-1}x}\big[\Psi_{\eps^{-2}t}^{\beta_{\eps}, h_0}\big]}{\EE_{\eps^{-1}x}\big[\Psi_{s}^{\beta_{\eps}, 0}\big]} \right\}; \EE_{\eps^{-1}x}\big[\Psi_{s}^{\beta_{\eps}, 0}\big] > b_{\eps} \right] 
 \\
& \leq C_{\beta}^{(1)}C_{\beta}^{(2)}   b_{\eps} + C_{h_0}  C_{\beta}^{(3)}C_\beta^{(4)} a_\eps b_\eps^{-2}.
\end{aligned}
\end{equation*}
By choosing $b_{\eps} = a_{\eps}^{1/3} $, Lemma~\ref{le:condvar} is proved.

\bibliographystyle{plain}
\bibliography{2dmesoscopicref}

\end{document}